\documentclass[a4paper, 11pt, reqno]{amsart}

\usepackage[margin=1in]{geometry}  
\usepackage{graphicx}               
\usepackage{amsmath}              
\usepackage{amsfonts}              
\usepackage{bm}                        
\usepackage{amsthm}                
\usepackage{mathtools}             
\usepackage{esint}                     
\usepackage{color}
\usepackage{amssymb}             
\usepackage[colorlinks]{hyperref}
\usepackage[noabbrev]{cleveref}
\crefname{subsection}{subsection}{subsections}
\usepackage{tikz}                       
\usepackage{pgfplots}                
\usetikzlibrary{patterns}              
\usepackage[english]{babel}
\usepackage[latin1]{inputenc}
\usepackage[T1]{fontenc}
\usepackage{soul}                     
\usepackage{enumitem}

\newtheorem{thm}{Theorem}[section]
\newtheorem{lem}[thm]{Lemma}
\newtheorem{prop}[thm]{Proposition}
\newtheorem{cor}[thm]{Corollary}

\newtheorem{rmk}[thm]{Remark}

\DeclareMathOperator{\tr}{Tr}

\DeclareMathOperator{\di}{div}

\newcommand{\RR}{\mathbb{R}}     
\newcommand{\NN}{\mathbb{N}}     
\newcommand{\B}{\mathcal{B}}
\newcommand{\D}{\mathcal{D}}

\newcommand{\J}{\mathcal{J}}

\newcommand{\e}{\varepsilon} 
\newcommand{\Dsy}{\mathbb{D}}

\newcommand{\abs}[1]{|{#1}|}

\hypersetup{colorlinks = true, linkcolor = blue, citecolor = red}
         
\numberwithin{equation}{section} 

\newcommand{\pder}[2]{\frac{\partial #1}{\partial #2}}

\newcommand{\I}{\mathbb{I}}
\newcommand{\F}{\mathbb{F}}
\newcommand{\Bb}{\mathbb{B}}

\keywords{Fluid-structure interaction, collision, bouncing.}
\subjclass[2020]{74F10, 76D07, 76M10.}
\date{\today}

\begin{document}


\title[] {Contactless rebound of elastic bodies in a viscous incompressible fluid}
\author[] {Giovanni Gravina} 
\address{Department of Mathematical Analysis\\Faculty of Mathematics and Physics\\ 
Charles University\\Prague\\ Czech Republic}
\email {gravina@karlin.mff.cuni.cz} 

\author[] {Sebastian Schwarzacher}
\address{Department of Mathematical Analysis\\Faculty of Mathematics and Physics\\ 
Charles University\\Prague\\ Czech Republic}
\email {schwarz@karlin.mff.cuni.cz}

\author[] {Ond\v{r}ej Sou\v{c}ek}
\address{Mathematical Institute\\Faculty of Mathematics and Physics\\ 
Charles University\\Prague\\ Czech Republic}
\email {soucek@karel.troja.mff.cuni.cz}

\author[] {Karel T\r{u}ma}
\address{Mathematical Institute\\Faculty of Mathematics and Physics\\ 
Charles University\\Prague\\ Czech Republic}
\email {ktuma@karlin.mff.cuni.cz} 

\maketitle

\begin{abstract}
In this paper, we investigate the phenomenon of particle rebound in a viscous incompressible fluid environment. We focus on the important case of no-slip boundary conditions, for which it is by now classical that, under certain assumptions, collisions cannot occur in finite time. Motivated by the desire to understand this fascinating yet counterintuitive fluid-structure interaction, we introduce a reduced model which we study both analytically and numerically. In this simplified framework, we provide conditions which allow to prove that rebound is possible even in the absence of a topological contact. Our results lead to conjecture that a qualitative change in the shape of the solid is necessary for obtaining a physically meaningful rebound. We support the conjecture by also comparing numerical simulations performed for the reduced model with the finite element solutions obtained for the corresponding well-established PDE system. 


\end{abstract}
\tableofcontents

\section{Introduction}

The problem of particle-particle or particle-wall collisions in viscous fluids has important practical applications and has thus been the subject of a plethora of studies, not only experimental and numerical, but also from a purely mathematical standpoint. Yet, the problem is far from being fully resolved and the partial results which are available can often be counterintuitive. An example is given by the simple case of a spherical rigid particle surrounded by a Stokes linear fluid which moves towards a wall. Indeed, it has been known for some time already that if both the particle and the wall are equipped with no-slip boundary conditions, contact cannot take place in a viscous incompressible fluid in finite time without singular forcing. See \cite{Brenner1961, cooley_o-neill-1969, Feireisl2004, Gerard-Varet2014, Gerard-Varet2015, GraHil16, Hesla, MR2354496, HilSekSok18, MR2481302, MR2226121} for an exhaustive and quantitative analysis on that question. 
Despite the fact that in an incompressible fluid with no-slip boundary conditions the contact seems to be impossible, it has been hypothesized that the particle can rebound provided that it is elastic, or in general, when it admits the storage and release of mechanical energy during the rebound, see e.g.\@ \cite{Davis1986}. 

Apart from that, other physical mechanisms allowing for contact or rebound have been suggested and investigated, such as slip boundary conditions~\cite{Gerard-Varet2014}, the fluid compressibility \cite{Feireisl2003a}, pressure-dependent material properties \cite{barnocky-davis-1989}, wall roughness \cite{Gerard-Varet2015}, etc.\@ (see also \cite{gonzalez-phd-2003} or \cite{joseph_zenit_hunt_rosenwinkel_2001} and references therein). \\

In this paper we consider a solid object (also referred to as particle or structure) that may be elastic or rigid and study its motion when thrown towards a rigid wall in a viscous incompressible liquid environment that adheres to all surfaces (that is, under no-slip boundary conditions). We consider both the two or three dimensional case. For simplicity we will assume that the fluid is governed by the steady Stokes equations \eqref{QSS}. We expect, however, that most of our observations should be also relevant in case the fluid is governed by the steady or unsteady Navier--Stokes equation, as it was observed by other authors for similar questions (see e.g.\@ \cite{Gerard-Varet2014, Gerard-Varet2015}).

As is explained and quantified in all detail in the given references (see e.g.\@ \cite{MR2354496, Hesla,MR2481302}) it is mathematically proven that the interplay of the {\em regularity of the surface}, the {\em incompressibility} and the {\em no-slip} boundary condition of the fluid imply that the a smooth, rigid body will not reach any other solid obstacle in finite time. This phenomenon is also known as the no-contact paradox. See \Cref{Fig:rigid_FEM} for a demonstration of this phenomenon in the case of ball falling towards a flat horizontal ramp.

\begin{figure}[!h]
\begin{center}
\includegraphics[width=5cm]{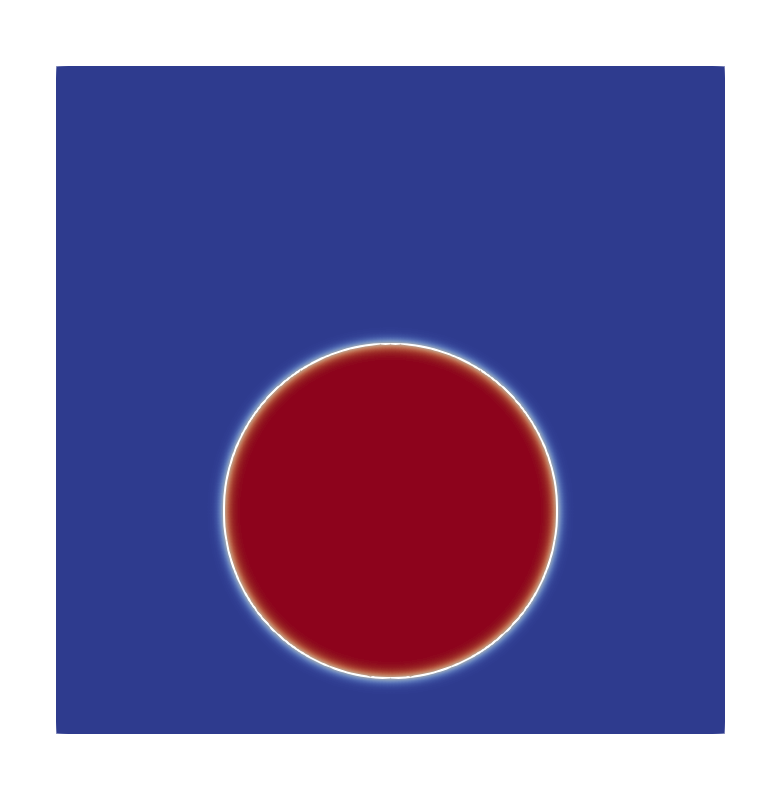}\hspace*{1.4cm}
\includegraphics[width=7.3cm]{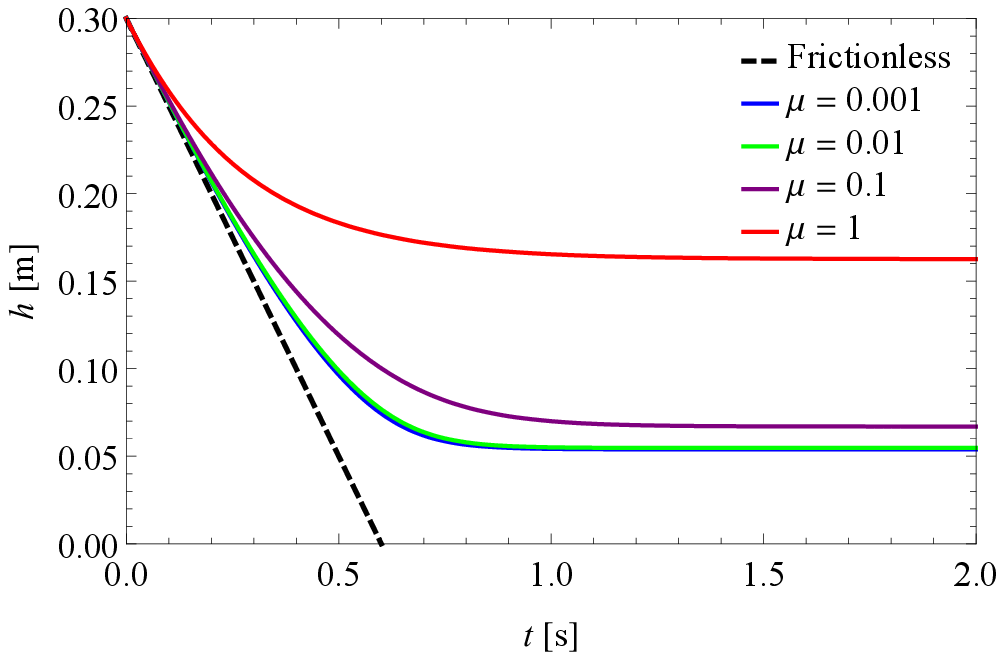}\\
\hspace*{-0.8cm}(a)\hspace*{7.5cm}(b)
\caption{(a) Rigid ball does not touch the bottom. Viscosity $\mu = 0.1$ Pa s, cf.\@ \Cref{fig_velocity} computed with the elastic ball whose deformation enables to get much closer to the boundary. (b) Dependence of $h$ on time $t$ for different viscosities, cf.\@ \Cref{fig:miracle}(b) for elastic ball.} \label{Fig:rigid_FEM}
\end{center}
\end{figure}

In this paper, we aim to advance the understanding of the extent to which the pathological behavior described in the no-contact paradox can affect the dynamics of solid particles in close proximity to the boundary of the container. Throughout the paper, special emphasis is given to the phenomenon of particle rebound. Indeed, the main question that motivated this work can be formulated as follows:

\begin{enumerate}[label=(Q.\arabic*), ref=Q.\arabic*]
\item \label{Q1} {\em Can solid particles rebound in the absence of a topological contact?}
\end{enumerate}

One of our main contributions is that we provide an affirmative answer to \eqref{Q1} in a simplified setting. To be precise, we introduce a system of coupled non-linear ODEs as a toy model approximation for the notoriously challenging fluid-structure interaction problem describing the motion of an \emph{elastic} solid immersed in a viscous incompressible fluid. 

Our design of the reduced model is methodologically inspired by the observation that, under certain simplifying assumptions, the motion of a \emph{rigid} body (described by the coupled fluid-structure interaction PDE system) can be reduced to a single second order ODE (see \cite{MR2354496}; see also \Cref{sec:rigid-body}). Conceptually it is inspired by numerical experiments (see Section~\ref{sec:FEM_results}). In particular, our simplified model (described in detail in \Cref{red-model-subsec}) presents the following two defining features:
\begin{itemize}
    \item[$(i)$] it allows for the storage and release of mechanical energy to account for an elastic response of the solid (see \Cref{fig-spring-mass-model}); 
    \item[$(ii)$] it encodes possible deformations of the body.
\end{itemize}       
While property $(i)$ is a rather natural requirement, a few comments on $(ii)$ are in order. Inspired by our numerical experiments, we allow the fluid-solid interaction to affect the shape of the solid object. It is well understood (see, for example, the discussion in \Cref{drag-section} and the reference therein), that changes in the flatness of the particle in the nearest-to-contact region can have a significant influence on the magnitude of the drag force. Furthermore, since this effect becomes even more dramatic at small distances from other solid objects or from the boundary of the container, we tailor our model to adequately capture this interplay by considering a possible dependence on the deformation parameter in the damping term which represents the drag force. In this simplified setting (see \Cref{mainB}), we show that rebound is indeed possible for sufficiently small values of the viscosity parameter, provided that the solid experiences a substantial flattening. \\

Let us mention here that our investigation uncovers a rather surprising ``trapping'' phenomenon, thus providing further insight into the consequences of the no-contact paradox. In order to illustrate this effect, consider a rigid object, which however allows for the storage and release of (a fraction of) its kinetic energy, as in property $(i)$ above. As a model example, consider a rigid spherical shell with an internal mass-spring energy absorbing mechanism, as sketched in \Cref{fig-spring-mass-model}, falling towards a horizontal wall. The expected dynamics for this particular configuration are as follows: as the outer shell is slowed-down by the viscous forces preventing from collision, part of the kinetic energy of the system is stored in the inner mechanism; the shell can then be expected to rebound once this energy is transferred back to it by the upwards push applied by the mass-spring system. Moreover, one would also anticipate to witness increasingly pronounced rebounds as friction in the fluid is reduced by considering gradually smaller values of the viscosity parameter. However, the analysis of this peculiar fluid-structure interaction performed on our reduced model predicts the following behavior.

\begin{cor}
\label{cor:1}
In the vanishing viscosity limit, the rigid shell system described above falls freely (that is, as it would in the vacuum) towards the wall, to which it then sticks for all times after collision.

\end{cor}
For a proof, we refer the reader to that of \Cref{mainNB} below, in which we show a more general result.

In view of \Cref{cor:1}, throughout the rest of the paper we say that a system does not produce a \emph{physical rebound} if the distance between the body and the wall converges, in the vanishing viscosity limit, to a monotone function in the time variable $t$. Thus, for our purposes, a rebound is said to be physical (or physically meaningful) if it withstands the vanishing viscosity limit.  

Obviously, some crucial aspect is missing in the models considered in \Cref{cor:1} (and \Cref{mainNB}) in order to capture physical bouncing effects.
Their motion is not only in clear contrast with our real-world experience of bouncing objects, but also with numerical simulations for elastic objects--where no-slip boundary conditions are imposed. See \Cref{sec:num} for a reference and the results presented in~\cite{frei2016eulerian,Ric17}. These observations naturally lead to the following question.
\begin{enumerate}[label=(Q.\arabic*), ref=Q.\arabic*]
\setcounter{enumi}{1}
\item \label{Q2} {\em What is the mathematical reason for a physical rebound?}
\end{enumerate}
We present here our scientific progress on this complicated issue. Specifically, our investigations and results prompted us to formulate the following conjecture.
\\

{\bf Conjecture:} {\em A qualitative change in the flatness of the solid body as it approaches the wall, together with some elastic energy storage mechanism within the body, allows for a physically meaningful rebound even for no-slip boundary conditions preventing topological contact.
}
\\

The results presented in this paper (both analytical and numerical) strongly support our leading conjecture. Indeed, it turns out that our ``educated guess'' in the design of the reduced model, for which we are able to prove the possibility of a physical rebound, admits solutions that are in striking match with the finite element solutions (FEM solutions) for a full fluid-structure interaction. Please see \Cref{fig:miracle} where the motions are compared for several values of the viscosity parameter. We refer to \Cref{sec:ODE-FEM_comparison} for a detailed discussion of the comparison between the numerical simulations.

\begin{figure}[!ht]
\begin{center}
\includegraphics[width=7.3cm]{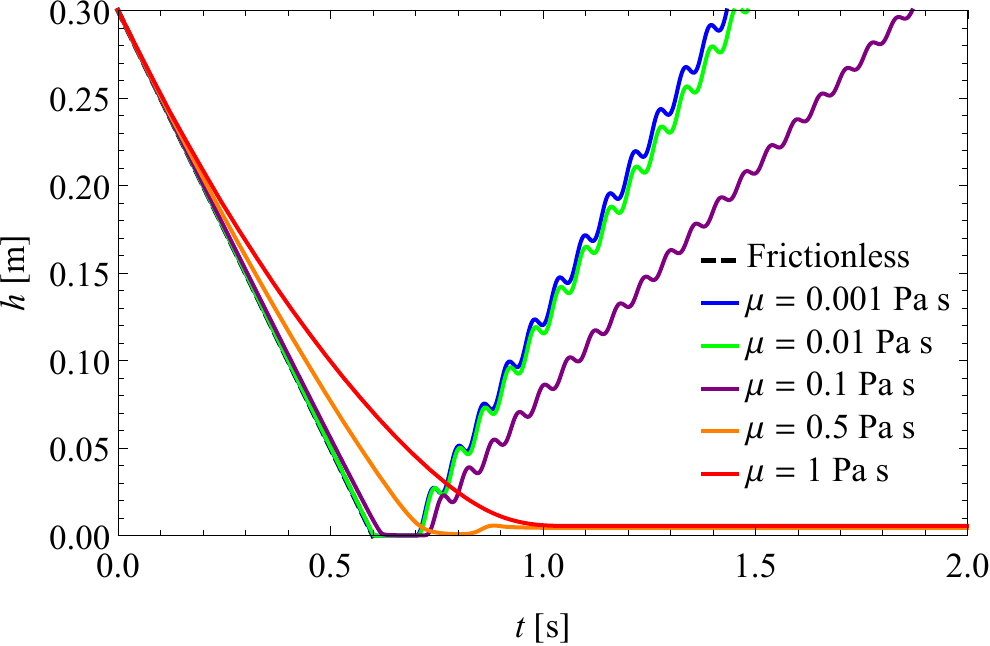}\hfill
\includegraphics[width=7.3cm]{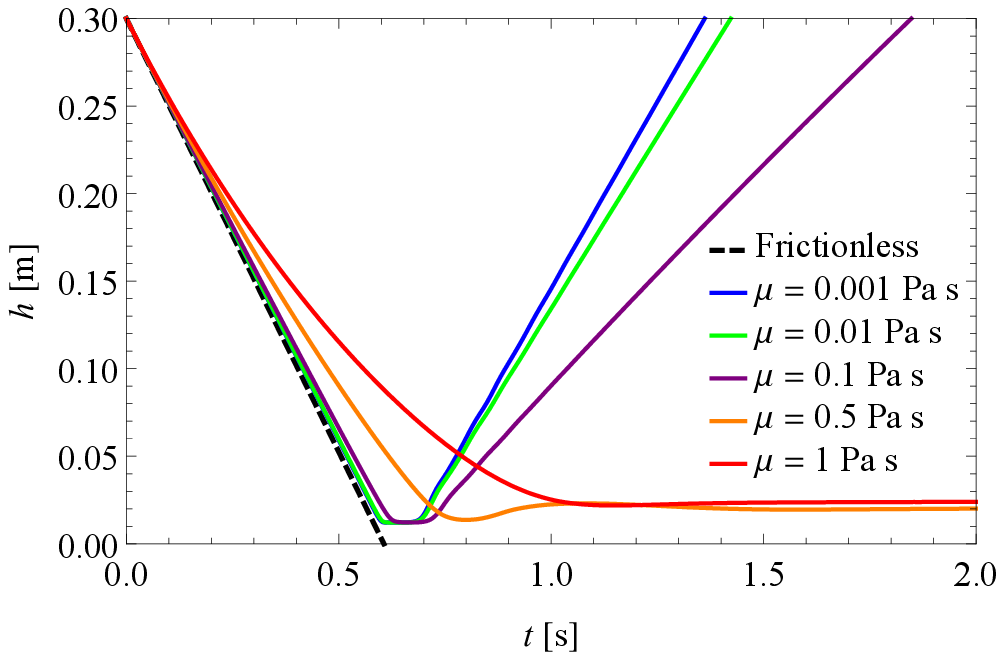}\\
\hspace*{0.7cm}(a)\hspace*{8.3cm}(b)
\caption{Comparison of the simple ODE solution (a) and FEM solution (b).} \label{fig:miracle}
\end{center}
\end{figure}

While this figure allows to speculate that our reduced model could have indeed potentially captured the essential feature for rebound in the absence of collisions, certainly, a precise connection between the models is still missing.
It is worth noting, however, that up to now even the existence theory for bulk elastic solids interacting with fluids is sparse rather sparse (see, for example, \cite{BenKamSch20,grandmont2002existence}). On the other hand, no-contact results (which can be regarded as the starting point of our investigations) for smooth deformable objects can be expected to be true. An important result in this direction is given by the paper \cite{GraHil16}, where the authors consider the case of a beam interacting with a viscous fluid. 

Special effort is put into keeping the assumptions in the analytical section of the paper as general as possible, without hindering its tractability. For this reason, in \Cref{mainresults} we provide an axiomatic set of assumptions which give the reduced model enough flexibility when it comes to fitting it with the full FSI problem.

\subsection{Structure of the paper}
The paper is organized as follows. In \Cref{section:full-FSI} we start by introducing the full fluid-structure interaction model, which we used for our numerical experiments. This is followed by the introduction of our reduced model of ODEs. The section is closed by the derivation of drag-formulas for the family of deformations that we consider for our numerical experiments and which forms the model case for the analysis. \Cref{sec:math} is dedicated to the main mathematical results of this paper and their proofs. In \Cref{mainresults} we introduce our general assumptions and state the main theorems. In particular, we provide conditions that allow to prove or disprove rebound in the vanishing viscosity limit. \Cref{proof-sec} is dedicated to the proofs of these results. In \Cref{sec:num}, we first provide numerical experiments for the reduced model of ODEs. In the following subsection we introduce the numerical set up that allows to capture the bouncing behavior of elastic solids for small viscosities and provide some numerical experiments. We conclude the section with the comparison from a numerical standpoint of the ODE and PDE solutions (see \Cref{fig:miracle}). Finally, in \Cref{ref:conclusion} we summarize and discuss our results.

\section{Modeling of particle-wall approach and rebound in viscous fluids}
In this section, we collect the various models employed throughout the paper for the study of near-to-contact dynamics. 

\subsection{The viscous fluid -- elastic structure formulation}
\label{section:full-FSI}
Consider an incompressible Newtonian fluid filling the region $\mathcal{F}(t)$, which surrounds an elastic particle whose position, at time $t$, will be denoted by $\mathcal{B}(t)$. For simplicity, we assume that the system composed by the fluid and the solid body occupies the entire half-space $\RR^N_+$, that is, $\overline{\mathcal{F}(t) \cup \mathcal{B}(t)} = \{\bm{x} \in \RR^N,x_N\geq 0\}$, where $N = 2$ or $N = 3$. As it is customary in fluid mechanics, the balance equations of linear momentum for the fluid are given in the Eulerian reference frame and read as follows:
\[
\begin{split}
\di \bm{v} &= 0,\\
\rho_f \left(\frac{\partial \bm{v}}{\partial t} + \bm{v} \cdot \nabla \bm{v}\right) & = \di \sigma_f + \rho_f \bm{b},
\end{split} \hspace{1cm} \text{ in } \mathcal{F}(t),
\]
where $\bm{v}(\bm{x},t)$ is the fluid velocity, $\rho_f$ is the constant fluid density, and $\bm{b}$ represents external bulk forces. Here the variable $\bm{x}$ denotes a position in the current (Eulerian) configuration, that is, $\bm{x} \in \mathcal{F}(t)$. We recall that for Newtonian fluids the Cauchy stress tensor takes the form
\[
\sigma_f = -p \mathbb{I}_N + 2 \mu \mathbb{D}(\bm{v}),
\]
where $p$ denotes the pressure, $\mathbb{I}_N$ is the $N$-dimensional identity matrix, $\mu$ is the constant dynamic viscosity, and $\mathbb{D}(\bm{v}) \coloneqq \frac{1}{2}(\nabla \bm{v} + (\nabla \bm{v})^{\rm T})$ is the symmetric part of the gradient of $\bm{v}$. On the other hand, the balance equations for the elastic solid are given in the Lagrangian setting and can be written as
\[
\begin{split}
\rho_s \frac{\partial^2 \bm{\eta}}{\partial t^2} &= \di (J \sigma_s \mathbb{F}^{-\rm T}) + \rho_s \bm{b},\\
J\rho_s &= \rho^0_s
\end{split}
\hspace{1cm} \text{ in } \B_0,
\]
where $\rho_s$ and $\rho_s^0$ denote the density of the elastic solid at time $t$ and in the reference configuration, respectively, $\bm{\eta}(\bm{X},t)$ is the displacement, $\mathbb{F}(\bm{X},t) \coloneqq \nabla_{{\bm X}}{\bm x}({\bm X},t) = \mathbb{I}_N + \nabla_{\bm{X}}\bm{\eta}(\bm{X},t)$ is the deformation gradient, $J \coloneqq \mathrm{det}\,\mathbb{F}$, and finally, $\B_0$ is the reference configuration of the solid. Here the variable $\bm{X}$ denotes a position in the reference (Lagrangian) configuration, that is, $\bm{X} \in \B_0$. We assume that the structure is an incompressible hyperelastic solid, i.e.\@,
\[
J \sigma_s \mathbb{F}^{-\rm T} = \frac{\partial \mathcal{L}}{\partial \mathbb{F}},\qquad
\mathcal{L}(\F,\tilde{p}) \coloneqq \mathcal{W}(\F)-\tilde{p}(J-1),
\]
where $\mathcal{L}$ is the Lagrange function corresponding to the strain energy function $\mathcal{W}$ under the restriction $J = 1$ and $\tilde{p}$ is the associated Lagrange multiplier. It is possible to use different strain energies $\mathcal{W}$ corresponding to different elastic models. As a particular example used later in the numerical computations (see \Cref{sec:num}), we consider an incompressible neo-Hookean solid with elastic strain energy given by
\[
\mathcal{W} \coloneqq \frac{G}{2}(|\F|^2-N).
\]
As one can readily check, in this case the Cauchy stress takes the form
\begin{equation}\label{stress-NeoHooke}
\sigma_s = -\tilde{p} \I_N + \frac{1}{J}\pder{\mathcal{W}}{\F}\F^{\rm T} = -\tilde{p} \I_N + G\Bb = -p\I_N + G\Bb^d,
\end{equation}
where $\Bb^d \coloneqq \Bb-(1/N)(\tr\Bb)\I_N$ is the deviatoric part of the left Cauchy-Green tensor $\Bb \coloneqq \F\F^{\rm T}$ and
$p \coloneqq \tilde{p}-(1/N)\tr\Bb$.

The conditions describing the interaction between the fluid and the solid comprise the continuity of the velocities and of the tractions:
\begin{align}\label{interaction_conditions}
\begin{split}
\bm{v}(\bm{x},t) &= \frac{\partial \bm{\eta}}{\partial t}(\bm{X},t),\\
\sigma_f \bm{n} &= \sigma_s \bm{n},
\end{split}
\hspace{1cm}\text{ on } \partial \B(t),
\end{align}
where $\bm{x} = \bm{X} + \bm{\eta}(\bm{X},t)$
and $\bm{n}$ is the unit normal to the fluid-solid interface. Finally, we prescribe no-slip boundary conditions on the boundary of the cavity, that is,
\[
\bm{v} = \bm{0} \hspace{1cm} \text{ on } \{x_N = 0\}
\]
and at infinity.

We remark that in \Cref{sec:num} we reformulate the mixed Lagrangian--Eulerian problem fully in the Eulerian frame. This allows for an efficient numerical implementation by finite element methods using a level-set function approach.
 
\subsection{Reduced models}
In view of the analytical challenges posed by the full FSI system described in \Cref{section:full-FSI}, in this paper we propose a simplified model which we believe to adequately capture the essential features of the FSI phenomena under consideration, with special emphasis on the questions of contact and rebound. This is achieved via a two-step procedure. First, we consider a completely rigid particle and show that, under certain simplifying assumptions, its dynamics can be replaced by a single ODE. As a next step, we enrich the model by taking into account possible elastic deformations of the particle, which we approximate by a single scalar internal degree of freedom. In our simplified framework, this internal variable will be used to parameterize not only the change in shape of the particle (which will be reflected in the expression for the drag force, see \Cref{drag-section} below), but also its elastic response. The final reduced model takes the form of two coupled ODEs with a highly non-linear damping term.
\subsubsection{Dynamics of a rigid body as a second order ODE with non-linear damping}
\label{sec:rigid-body}
In this section we show that, under certain assumptions, the dynamics of a rigid body in a viscous incompressible fluid can be reformulated as a second order non-linear ODE, which takes the the form 
\[
\ddot h = - d(h)\dot h.
\]
To be precise, following the approach of Hillairet (see Section 3 in \cite{MR2354496}), we assume that the system composed by the fluid and the rigid body occupies the entire half-space $\RR^N_+$, $N = 2, 3$, and that the fluid adapts instantaneously to the solid, so that it can be effectively modeled by the quasi-static Stokes equations. Furthermore, if we suppose that the range of possible motions of the body consists only of translations in the direction $\bm{e}_N$, its position is uniquely determined by its distance from the set $\{x_N = 0\}$, denoted here and in the following with $h$. Let $\B \subset \RR^N_+$ denote the bounded region occupied by the rigid body when $h = 0$ and define
\begin{equation}
\label{fluid-dom}
\B_h \coloneqq \B + h \bm{e}_N, \qquad \mathcal{F}_h \coloneqq \RR^N_+ \setminus \B_h.
\end{equation}
With these notations at hand, and under the assumption that the fluid is homogeneous with density $\rho_f = 1$, our fluid-structure interaction problem is described by the balance equations of linear momentum, which read as
\begin{equation}
\label{QSS}
\left\{
\arraycolsep=1.4pt\def\arraystretch{1.6}
\begin{array}{rll}
- \mu \Delta \bm{v} + \nabla p = & \bm{0} & \text{ in } \mathcal{F}_{h}, \\
\di \bm{v} = & 0 & \text{ in } \mathcal{F}_{h}, \\
\bm{v} = & \dot h \bm{e}_N &  \text{ on } \partial \B_{h}, \\
\bm{v} = & \bm{0} &  \text{ on } \{x_N = 0\}, \\
\bm{v} = & \bm{0} &  \text { at } \infty,
\end{array}
\right.
\end{equation}
coupled with the continuity of the stresses across the fluid-solid surface, which in the present framework can be expressed via
\begin{equation}
\label{QSSS}
m \ddot h = - \int_{\partial \B_h} \left(2\mu \Dsy \bm{v} - p \mathbb{I}_N\right)\bm{n} \,d\mathcal{H}^{N - 1} \cdot \bm{e}_N.
\end{equation}
We recall that, as in the previous subsection, we use $\bm{v}$ and $p$ to denote the velocity field and the pressure of the fluid, respectively. Moreover, the positive constants $\mu$ and $m$ represent the viscosity of the fluid and the mass of the body, respectively. Finally, throughout the section $\bm{n}$ is always used to denote the outer unit normal vector to the fluid domain. The system \eqref{QSS}--\eqref{QSSS} is further complemented with initial conditions of the form
\[
h(0) = h_0 > 0, \quad \dot h(0) = \dot h_0.
\]
The next result combines Lemma 4 and Lemma 5 in \cite{MR2354496}.
\begin{lem}
\label{pLM}
Let $h > 0$ be given and assume that $\partial \B$ is Lipschitz continuous. Then there exist a unique velocity field $\bm{s}_h$ and a pressure field $\pi_h$ such that
\begin{equation}
\label{sh}
\left\{
\arraycolsep=1.4pt\def\arraystretch{1.6}
\begin{array}{rll}
-\Delta \bm{s}_h + \nabla \pi_h = & \bm{0} & \text{ in } \mathcal{F}_h, \\
\di \bm{s}_h = & 0 & \text{ in } \mathcal{F}_h, \\
\bm{s}_h = & \bm{e}_N &  \text{ on } \partial \B_h, \\
\bm{s}_h = & \bm{0} &  \text { on } \{x_N = 0\}, \\
\bm{s}_h = & \bm{0} &  \text{ at } \infty.
\end{array}
\right.
\end{equation}
Moreover, the following statements hold:
\begin{itemize}
\item[$(i)$] $\bm{s}_h$ is the unique global minimizer for the functional 
\begin{equation}
\label{J-def}
\J(\bm{u}; \mathcal{F}_h) \coloneqq \int_{\mathcal{F}_h}|\Dsy \bm{u}|^2\,dx,
\end{equation}
defined over the class 
\[
V_h \coloneqq \left\{\bm{u} \in H^1_0(\RR^N_+;\RR^N) : \di \bm{u} = 0, \text{ and } \bm{u} = \bm{e}_N \text{ on } \partial \B_h \right\}.
\]
In particular, the pressure function $\pi_h$ can be understood as the Lagrange multiplier associated to the divergence-free constraint in $V_h$.
\item[$(ii)$] For every $\tilde{\bm{\varphi}} \in V_h$ and $z \in \RR$, if we let $\bm{\varphi} \coloneqq z\tilde{\bm{\varphi}}$ we have
\begin{equation}
\label{weak-sh}
2 \int_{\mathcal{F}_h}\Dsy \bm{s}_h : \Dsy \bm{\varphi} \,dx = \int_{\partial \B_h}\left( 2\Dsy \bm{s}_h - \pi_h \mathbb{I}_N\right)\bm{n} \,d\mathcal{H}^{N - 1} \cdot z \bm{e}_N.
\end{equation}
\item[$(iii)$] The function $\bm{s}_h$ depends smoothly on the parameter $h$, for all $h \in (0, \infty)$.
\end{itemize}
\end{lem}

As a consequence of \Cref{pLM} we see that the dynamics of the system are fully characterized by an initial value problem for a second order ODE with a non-linear damping term. 

\begin{lem}
\label{pde-ode}
Assume that $\partial \B$ is Lipschitz continuous. Then, for every $h_0 > 0$ and $\dot h_0 \in \RR$, the solvability of the fluid-structure interaction problem \eqref{QSS}--\eqref{QSSS} reduces to that of the initial value problem 
\begin{equation}
\label{reduced-pb}
\left\{
\arraycolsep=1.4pt\def\arraystretch{1.6}
\begin{array}{l}
m \ddot h = - \mu \J(\bm{s}_h; \mathcal{F}_h) \dot h , \\
h(0) = h_0,\ \dot h(0) = \dot h_0.
\end{array}
\right.
\end{equation}
\end{lem}

\begin{proof} Notice that for any given $h > 0$ and $\dot h \in \RR$, letting $\bm{v} \coloneqq \dot h \bm{s}_h$ and $p \coloneqq \mu \dot h \pi_h$ yields a solution to \eqref{QSS}. Moreover, using $\bm{\varphi} \coloneqq \mu \dot h \bm{s}_h$ as a test function in \eqref{weak-sh}, we obtain 
\begin{align}
2\mu \dot h \int_{\mathcal{F}_h} |\Dsy \bm{s}_h|^2\,dx & = \int_{\partial \B_h}\left( 2\Dsy \bm{s}_h - \pi_h \mathbb{I}_N\right)\bm{n} \,d\mathcal{H}^{N - 1} \cdot \mu \dot h \bm{e}_N \notag \\
& = \int_{\partial \B_h}\left( 2 \mu \Dsy \bm{v} - p \mathbb{I}_N\right)\bm{n} \,d\mathcal{H}^{N - 1} \cdot \bm{e}_N. \label{weak-sh-test}
\end{align}
In view of \eqref{weak-sh-test}, we can then rewrite \eqref{QSSS} as 
\[
m \ddot h = - 2\mu \dot h \int_{\mathcal{F}_h} |\Dsy \bm{s}_h|^2\,dx = - \mu \J(\bm{s}_h; \mathcal{F}_h) \dot h .
\]
This concludes the proof.
\end{proof}

\subsubsection{Spring-mass model}
\label{red-model-subsec}
In this subsection, we enrich the model described in \eqref{QSS}--\eqref{QSSS} by considering also elastic deformations of the particle. As a first approximation, we will assume that the deformation of the particle can be described by a single scalar parameter $\xi$, which we can think of as the deformation of an internal spring with stiffness $k$ carrying internal mass $m$, enclosed in a shell of mass $M$ which is rigid with respect to the flow of surrounding fluid, but whose shape may change according to the value of the internal parameter $\xi$ (the relevant notation is summarized in \Cref{fig-spring-mass-model}; see \Cref{rebound-cartoon} for a schematic illustration of contactless rebound for the case of a deformable particle). 

To be precise, let $\mathcal{P}$ denote the class of all admissible particle configurations, that is, $\mathcal{P}$ is the family of all bounded open subsets of $\RR^N$ with Lipschitz continuous boundary and such that the intersection of their respective closures with the hyperplane $\{x_N = 0\}$ consists of only the origin. Given $\B \in \mathcal{P}$, we consider a one parameter family of diffeomorphisms $\{G_{\xi} \colon \B \to G_{\xi}(\B) : \xi \in \RR\}$ such that $G_{\xi}(\B) \in \mathcal{P}$ for every $\xi \in \RR$. Moreover, for every $h > 0$ and every $\xi \in \RR$, we let 
\[
\mathcal{F}_{h,\xi} \coloneqq \RR^N_+ \setminus (G_{\xi}(\B) + h \bm{e}_N)
\]
and consider the energy functional
\[
\J(\bm{u}; \mathcal{F}_{h, \xi}) \coloneqq \int_{\mathcal{F}_{h,\xi}} |\Dsy \bm{u}|^2\,dx.
\]
Compare these definitions with their counterparts in the previous subsection, i.e.\@ \eqref{fluid-dom} and \eqref{J-def}, respectively. In particular, by an application of \Cref{pLM}, we obtain that for each $h > 0$ and each $\xi \in \RR$ there exists a vector field $\bm{s}_{h, \xi}$ that minimizes $\J(\cdot; \mathcal{F}_{h, \xi})$ over the class
\[
V_{h, \xi} \coloneqq \left\{\bm{u} \in H^1_0(\RR^N_+;\RR^N) : \di \bm{u} = 0, \text{ and } \bm{u} = \bm{e}_N \text{ on }  G_{\xi}(\partial \B) + h \bm{e}_N \right\}.
\]
\begin{center}
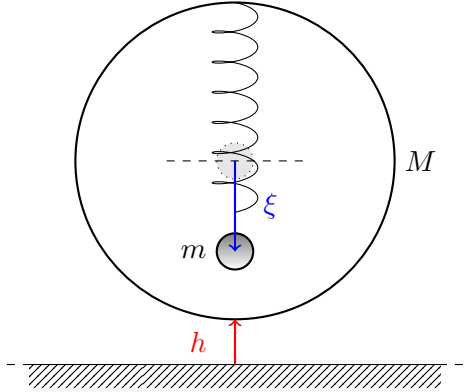
\begin{figure*}[h]\centering
\begin{tikzpicture}[scale=0.6]
\draw [dashed] (0, 0) -- (0.5,0);
\draw (0.5, 0) -- (9.5, 0);
\draw [dashed] (10, 0) -- (9.5,0);
\path [pattern=north east lines, pattern color=black, thin] (0.5,0) rectangle (9.5,-0.5);
\draw [thick] (5,4.5) circle (3.5cm);
\node [left, red] at (4.6, 0.5) {$h$};
\node [right] at (8.5, 4.5) {$M$};
\draw [dotted, fill = black, fill opacity = 0.1] (5,4.5) circle (0.4cm);
\shadedraw [thick] (5,2.5) circle (0.4cm);
\draw[decoration={aspect=0.3, segment length=4mm, amplitude=3mm,coil},decorate] (5,8) -- (5,2.9); 
\node [left] at (4.6, 2.5) {$m$};
\draw [dashed, thin] (3.5, 4.5) -- (6.5, 4.5);
\draw [->, thick, red] (5, 0) -- (5, 1);
\draw [->, blue, thick] (5, 4.5) -- (5, 2.5);
\node [right, blue] at (5.4, 3.5) {$\xi$};
\end{tikzpicture}
\caption{A spherical shell with an inner mass-spring system is surrounded by a viscous incompressible fluid.}
\label{fig-spring-mass-model}
\end{figure*}
\end{center}
Under the assumption that the range of possible motions of the deformable shell consists only of translations in the direction $\bm{e}_N$, reasoning as in \Cref{pde-ode} we see that its dynamics can be formulated as a second order ODE, with the exception that now at each time level $t$, the shape of the shell may change (depending on the value of $\xi$). Consequently, the mechanical force balance for such a system takes the form of the following system of two coupled ODEs:
\begin{align}
\label{red-model2}
M \ddot h &= - k\xi - \mu \J(\bm{s}_{h,\xi}, \mathcal{F}_{h, \xi})\dot{h}\,,\\
m( \ddot h - \ddot \xi) & = k\xi
\label{red-model1}
\end{align}
with initial conditions  
\begin{align*}
h(0) &= h_0,\hspace{1cm}\dot{h}(0)= \dot{h}_0,\\
\xi(0) &= \xi_0,\hspace{1cm}\dot{\xi}(0)= \dot{\xi}_0.
\end{align*}
We remark that equation \eqref{red-model2} expresses the dynamics of the internal mass-spring system in the frame accelerating with the outer shell, while the second equation \eqref{red-model1} is the analogue of \eqref{reduced-pb}, where the additional ``internal'' force is acting on the outer shell and with a more general drag force term which depends not only $h$, but also on the internal deformation $\xi$.

\begin{center}
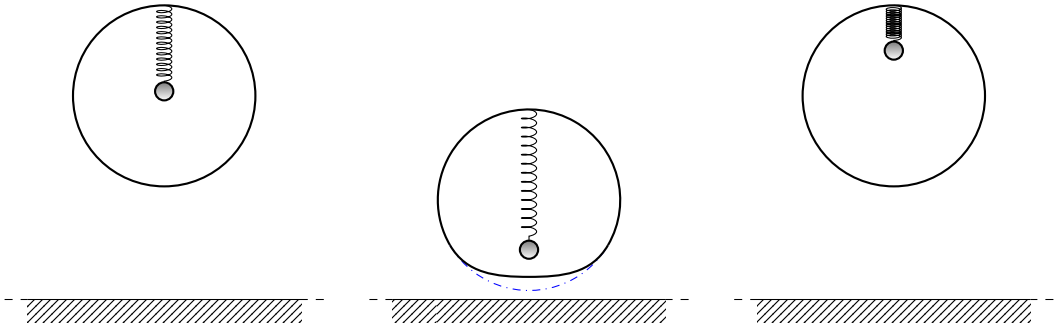
\begin{figure*}[h]\centering
\begin{tikzpicture}[scale=0.6]

\draw [dashed] (0, 0) -- (0.5,0);
\draw (0.5, 0) -- (6.5, 0);
\draw [dashed] (7, 0) -- (6.5,0);
\path [pattern=north east lines, pattern color=black, thin] (0.5,0) rectangle (6.5,-0.5);
\draw [thick] (3.5,4.5) circle (2cm);
\shadedraw [thick] (3.5,4.6) circle (0.2cm); 
\draw[decoration={aspect=0.3, segment length=0.7mm, amplitude=1mm,coil},decorate] (3.5,6.5) -- (3.5,4.8);

\draw [dashed] (8, 0) -- (8.5,0);
\draw (8.5, 0) -- (14.5, 0);
\draw [dashed] (15, 0) -- (14.5,0);
\path [pattern=north east lines, pattern color=black, thin] (8.5,0) rectangle (14.5,-0.5);
\draw [thick] (13.232,1.2) arc (-30:210:2);
\draw [dash dot, blue] (9.768,1.2) arc (210:330:2);
\draw [thick] (13.232,1.2) to [out=240 , in=0] (11.5,0.5) to [out=180, in=-60] (9.768, 1.2);
\shadedraw [thick] (11.5, 1.1) circle (0.2cm);
\draw[decoration={aspect=0.3, segment length=1.2mm, amplitude=1mm,coil},decorate] (11.5, 4.2) -- (11.5,1.3); 

\draw [dashed] (16, 0) -- (16.5,0);
\draw (16.5, 0) -- (22.5, 0);
\draw [dashed] (23, 0) -- (22.5,0);
\path [pattern=north east lines, pattern color=black, thin] (16.5,0) rectangle (22.5,-0.5);
\draw [thick] (19.5,4.5) circle (2cm);
\shadedraw [thick] (19.5,5.5) circle (0.2cm);
\draw[decoration={aspect=0.3, segment length=0.3mm, amplitude=1mm,coil},decorate] (19.5,6.5) -- (19.5,5.7); 

\end{tikzpicture}
\caption{Schematic representation of contactless rebound for a deformable shell with an inner energy absorbing mechanism. The dash-dotted line represents the undeformed surface.}
\label{rebound-cartoon}
\end{figure*}
\end{center}

\subsection{The drag force}
\label{drag-section}
As a consequence of the ODE reformulation of the FSI problem provided in \eqref{reduced-pb} (resp.\@ \eqref{red-model1}), we see that the drag force exerted by the fluid on the solid body, i.e.\@ the term $- \mu \J(\bm{s}_h; \mathcal{F}_h) \dot h$ (resp.\@ $- \mu \J(\bm{s}_{h, \xi}; \mathcal{F}_{h, \xi}) \dot h$), can significantly influence the behavior of the system. Thus, in this section we collect some well known approximations of this force. In order to obtain a precise understanding of the near-to-contact dynamics, the focus of the section is on the dependence of $\J(\bm{s}_h; \mathcal{F}_h)$ on the parameter $h$, with special emphasis on the case $h \to 0^+$. We recall indeed that $h = 0$ corresponds to a collision of the body with the boundary of the container.

To be precise, in the following we present estimates of the drag formulas for both the two and three dimensional case. Furthermore, we compare them also with those resulting from the standard lubrication (Reynolds') approximation. For the purpose of this section, it is not restrictive to consider rigid particles. Additionally, in all cases we shall assume the particle is axi-symmetric with respect to the axis $x_N$ ($N = 2, 3$) and that the part of the boundary $\partial\mathcal{B}$ that is closer to the wall can be described in a neighborhood of the origin by a graph of the form
\begin{align}
\label{eq:drag1}
\psi(x_1) = \gamma |x_1|^{1 + \alpha} \hspace{0.5cm} \text{ if } N = 2,\hspace{1cm} \psi(x_1, x_2) = \gamma (x^2_1 + x_2^2)^\frac{1 + \alpha}{2}\hspace{0.5cm}\text{ if } N = 3.\hspace{1cm}
\end{align}

\subsubsection{Drag force estimates based on the variational formulation}
\label{drag-var}
We begin by noticing that, depending on the smoothness of the immersed particle, the drag force exerted by the viscous fluid can develop a singularity when the distance between the body and the boundary of the cavity tends to zero. This is made precise in the next result, which is due to Starovoitov (see Theorem 3.1 in \cite{MR2044583}). A proof of the theorem is included in \Cref{app-estimates-stokes} for the reader's convenience. 

\begin{thm}
\label{2&3D-LB}
Let $\B$ be an open bounded subset of $\RR^N_+$ with Lipschitz continuous boundary and such that $\partial \B \cap \{x_N = 0\}$ consists of only the origin. For  $\J$ and $\bm{s}_h$ given as in \Cref{pLM}, let $D \colon (0, \infty) \to (0, \infty)$ be defined via 
\[
D(h) \coloneqq \J(\bm{s}_h; \mathcal{F}_h).
\]
Then $D$ is locally Lipschitz continuous. Furthermore, the following statements hold:
\begin{itemize}
\item[$(i)$] if $N = 2$ and there are $\alpha, \gamma, r > 0$ such that in a neighborhood of the origin $\partial \B$ coincides with the graph of $\psi(x_1) \coloneqq \gamma |x_1|^{1 + \alpha}$, for $|x_1| < r$, then there exists a positive constant $c_1$ such that for all $0 < h \le r^{1 + \alpha}$
\[
D(h) \ge c_1h^{\frac{- 3\alpha}{1 + \alpha}};
\]
\item[$(ii)$] if $N = 3$ and there are $\alpha, \gamma, r > 0$ such that in a neighborhood of the origin $\partial \B$ coincides with the graph of $\psi(x_1,x_2) \coloneqq \gamma (x_1^2 + x_2^2)^{\frac{1 + \alpha}{2}}$, for $x_1^2 + x_2^2 < r^2$, then there exists a positive constant $c_2$ such that for all $0 < h \le r^{1 + \alpha}$
\[
D(h) \ge c_2h^{\frac{1 - 3 \alpha}{1 + \alpha}}.
\]
\end{itemize}
\end{thm}
Roughly speaking, \Cref{2&3D-LB} presents us with the crucial observation that the asymptotic behavior of $D$ is deeply connected to the regularity of $\partial \B$ in a neighborhood of the nearest point to the fixed boundary of the container. It is worth noting that since for every $t \in (-1,1)$ one has that
\[
\frac{t^2}{2} \le 1 - \sqrt{1 - t^2} \le t^2,
\]
an application of \Cref{2&3D-LB} with $\alpha = 1$ yields that if $N = 2$ and $\B$ is a disk then $D(h) \gtrsim h^{-3/2}$, while if $N = 3$ and $\B$ is a sphere then $D(h) \gtrsim h^{-1}$. In particular, as illustrated in Theorem 3.2 in \cite{MR2044583} (see also Theorem 3 in \cite{MR2354496}), one can then transform the differential equation obtained in \Cref{pde-ode} into a differential inequality; this, in turn, can be integrated to show that the rigid body cannot collide with the boundary of the container in finite time. 

It is worth noting that the proof of the no-collision result in the papers \cite{MR2592281,MR2354496,MR2481302}, where the fluid is modeled by the Navier--Stokes equations, relies on the construction of a good (localized) approximation of the solution to the associated Stokes problem. A particularly interesting corollary of these constructions is that the asymptotic lower bounds provided by \Cref{2&3D-LB} are, in most cases, optimal. To be precise, we have the following theorem (for more information, see also the discussion at the end of \Cref{app-estimates-stokes}).
\begin{thm}
\label{2&3D-UB}
Under the assumptions of \Cref{2&3D-LB}, there exist two positive constants $C_1, C_2$ such that for all $h$ sufficiently small
\begin{equation}
\label{drag-UB}
D(h) \le 
\left\{
\arraycolsep=1.4pt\def\arraystretch{1.6}
\begin{array}{ll}
\displaystyle C_1 h^{\frac{- 3\alpha}{1 + \alpha}} & \text{ if } N = 2, \\
\displaystyle C_2 h^{\frac{1 - 3 \alpha}{1 + \alpha}} &  \text{ if } N = 3 \text{ and } \alpha > 1/3, \\
\displaystyle C_2 |\log h| &  \text{ if } N = 3 \text{ and } \alpha = 1/3, \\
\displaystyle C_2 &  \text{ if } N = 3 \text{ and } \alpha < 1/3.
\end{array}
\right.
\end{equation}
\end{thm}

We conclude the section by observing that, in the present framework, if $\partial \B$ is sufficiently regular so that the body is prevented from colliding in finite time with the boundary of the container, then the system cannot produce a rebound.  
\begin{cor}
\label{rigid-no-rebound}
Let $h$ be a solution to \eqref{reduced-pb} with initial conditions $h_0 > 0$ and $\dot h_0 < 0$, and assume that $h(t) > 0$ for every $t > 0$. Then $h$ is a monotone function.
\end{cor}
\begin{proof}
Arguing by contradiction, assume that there are $\tau_1 < \tau_2$ such that $\dot h(\tau_1) = 0$ and $h(\tau_2) >  \tilde{h} \coloneqq h(\tau_1)$. Since $\min\{h(t) : t \in [0, \tau_2]\} > 0$ and by recalling that $D$ is locally Lipschitz continuous in $(0, \infty)$, we see that the initial value problem \eqref{reduced-pb} admits a unique solution in $[0, \tau_2]$, which must therefore agree with $h$. Notice, however, that $h$ is also the unique solution to the initial value problem satisfying \eqref{reduced-pb} on $[\tau_1,\tau_2]$ with initial conditions $h(\tau_1)=\tilde{h}$ and $\dot{h}(\tau_1)=0$. Consequently $h\equiv \tilde{h}$ on $[\tau_1,\tau_2]$, which contradicts $h(\tau_2)>\tilde{h}$.
\end{proof}

\subsubsection{Drag force estimates based on Reynolds' approximation}
\label{drag-reynolds-app}
Similarly to above, throughout the subsection we consider an axi-symmetric particle $\B$. In particular, if $\partial \B$ satisfies \eqref{eq:drag1}, then in a neighborhood of the nearest-to-contact point $\partial \B_h$ (see \eqref{fluid-dom}) can be conveniently described as the graph of
\begin{equation}
\label{gh+}
g(r) = h + \gamma r^{1 + \alpha},
\end{equation}
where $r$ denotes the distance from the symmetry axis. With this notation at hand and in view of the lubrication (Reynolds') approximation (see \Cref{appendix-Rey}), we obtain that the vertical component of the drag force exerted on the particle can be effectively estimated by
\begin{equation}
\label{eq-reynolds-drag}
F_{\operatorname{lub}} \coloneqq -12\mu\dot{h} 
\left\{
\arraycolsep=3.4pt\def\arraystretch{2.2}
\begin{array}{ll}
\displaystyle 2\int_{0}^\infty\int_{r}^\infty \frac{r'}{g(r')^3} \,dr'dr & \text{ if } N = 2, \\
\displaystyle \pi \int_{0}^\infty \int_{r}^\infty \frac{r r'}{g(r')^3} \,dr'dr & \text{ if } N = 3.
\end{array}\right.
\end{equation}
An exact comparison of the drag formulas in \eqref{eq-reynolds-drag} with the resulting expressions derived in \Cref{drag-var} is only possible for particular values of $\alpha$, for which the Reynolds based expression can be integrated analytically. In particular, assuming circular (when $N = 2$) or spherical (when $N = 3$) shape of the solid ball with radius $R$, we get
\[
g(r) \coloneqq h + R - \sqrt{R^2-r^2} \sim h + \frac{r^2}{2R^2}.
\]
Substituting $\alpha = 1$ and $\gamma = 1/(2R)$ into \eqref{eq-reynolds-drag} allows to analytically resolve the integrals, which ultimately yields $F_{\operatorname{lub}} = - \mu D_{\operatorname{lub}}(h) \dot h$, where
\begin{equation}
\label{eq:drag2}
D_{\operatorname{lub}}(h) \coloneqq 
\left\{
\arraycolsep=3.4pt\def\arraystretch{2.2}
\begin{array}{ll}
\displaystyle 3 \sqrt{2} \pi \left( \frac{R}{h} \right)^{\frac{3}{2}} & \text{ if } N = 2, \\
\displaystyle 6 \pi \frac{R^2}{h} & \text{ if } N = 3;
\end{array}
\right.
\end{equation}
see also eq.\@ (7-270) in \cite{Leal1992}, eq.\@ (2.18) in \cite{Brenner1961}, and and eq.\@ (1.1) in \cite{Cox1971}.

On the other hand, in order to compare the expressions for the drag force for other values of $\alpha$, we compute numerically the lubrication theory shape factor $D_{\operatorname{lub}}$ from \eqref{eq-reynolds-drag} and compare it with the analytical estimates in \eqref{drag-UB} in \Cref{fig-drag-force-comparison}. Note that the match is very good for the case $N = 2$ and reasonable for the case $N = 3$ at least in the vicinity of $\alpha = 1$, corresponding to the sphere.

\begin{figure}[h!]
\begin{tabular}{cc}
\includegraphics[width=13cm]{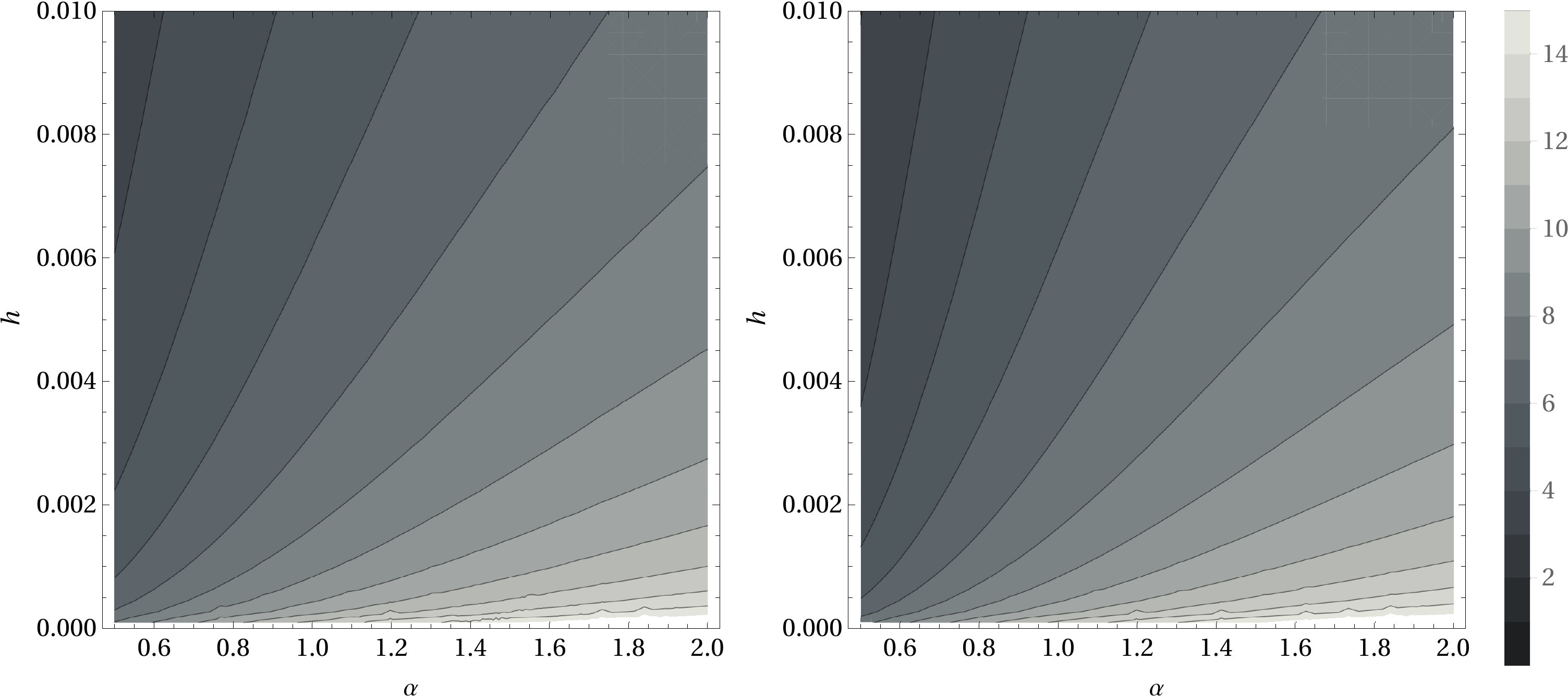}\\
\includegraphics[width=13cm]{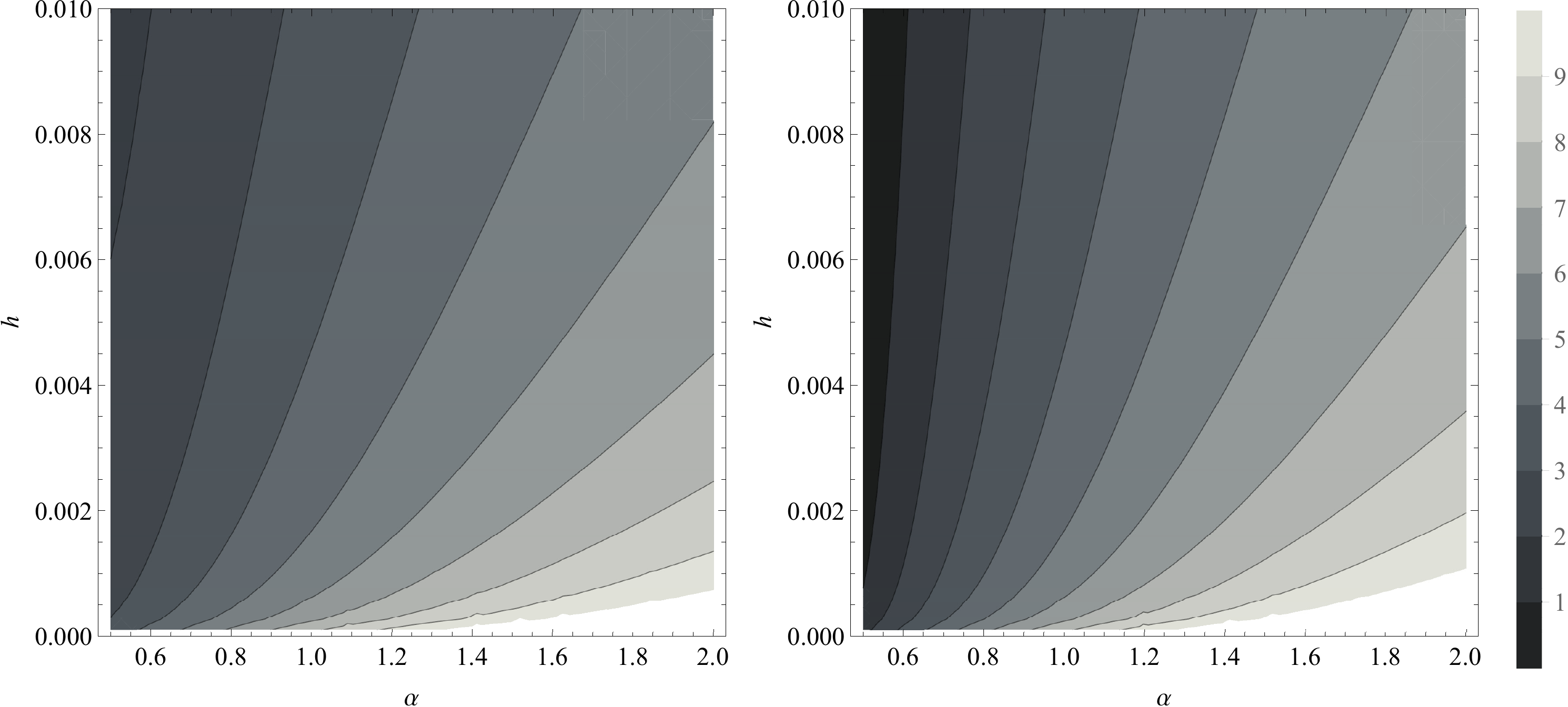}
\end{tabular}
\caption{Logarithm of the drag force shape factor based on the Reynolds approximation (left) and on the analytical estimate (right) for $N = 2$ (top row) and $N = 3$ (bottom row).}
\label{fig-drag-force-comparison}
\end{figure}

\section{Global well-posedness and qualitative behavior of solutions to the reduced model}
\label{sec:math}
In this is section, we undertake a rigorous analytical study of the reduced model that was previously introduced in \Cref{red-model-subsec}. We begin by addressing the question of global well-posedness and we then proceed to investigate qualitative properties of solutions as we vary the viscosity parameter $\mu$. In this direction, we present two results which highlight very different behaviors with regard to particle rebound. For clarity of exposition, we postpone the proofs to \Cref{proof-sec}. In addition, we refer the reader to \Cref{sec:ODE_num} for some numerical experiments on the model considered in this section.

\subsection{Statement of the main results}
\label{mainresults}
Throughout the section we consider the system of ODEs
\begin{equation}
\label{gf}
\left\{
\arraycolsep=1.4pt\def\arraystretch{1.6}
\begin{array}{l}
\ddot h - \ddot \xi  = a b(\xi), \\
\ddot h = - b(\xi) - \mu \D(h,\xi)\dot h, \\
h(0) = h_0,\ \dot h(0) = \dot h_0, \\
\xi(0) = \xi_0,\ \dot \xi(0) = \dot \xi_0.
\end{array}
\right.
\end{equation}
Here $a$ and $\mu$ are positive constants, while the functions $b$ and $\D$ serve as proxies for the elastic response of the solid and the drag force, respectively. Notice indeed that the system given by \eqref{red-model2}--\eqref{red-model1} is a particular case of \eqref{gf}, corresponding to the choices 
\begin{equation}
\label{choices-red}
b(\xi) \coloneqq \frac{k \xi}{M}, \qquad  a \coloneqq \frac{M}{m},\qquad  \D(h, \xi) \coloneqq \frac{\J(\bm{s}_{h, \xi}, \mathcal{F}_{h, \xi})}{M}.  
\end{equation}

In our first result, the aim is to identify conditions for which the body is prevented from colliding with the boundary of the container in finite time. Our analysis is in spirit very close to that of \cite{MR2354496} (see also \Cref{2&3D-LB} and the subsequent discussion). To this end, we define
\[
B(y) \coloneqq \int_0^y b(w)\,dw,
\] 
and make the following assumptions: 
\begin{enumerate}[label=(B.\arabic*), ref=B.\arabic*]
\item \label{B1} $b \colon \RR \to \RR$ is locally Lipschitz continuous;
\item \label{B2} $B$ is coercive, that is, $B(y) \to \infty$ as $|y| \to \infty$.
\end{enumerate}
Additionally, on $\D \colon (0, \infty) \times \RR \to (0, \infty)$ we require an analogous regularity condition and a singular asymptotic lower bound that is uniform with respect to the variable $\xi$. To be precise, throughout the following we always work under the following set of assumptions:
\begin{enumerate}[label=(D.\arabic*), ref=D.\arabic*]
\item \label{D1} the map $(h,\xi) \mapsto \D(h, \xi)$ is locally Lipschitz continuous in $(0, \infty) \times \RR$;
\item \label{D2} there exist a constant $c > 0$ and $\alpha \in [1,\infty)$ such that for all $h > 0$ and $\xi \in \RR$
\[
\D(h, \xi) \ge ch^{-\alpha}.
\]
\end{enumerate}
It is worth noting that the assumptions above are satisfied, for example, by the drag force exerted on a circular or spherical structure (see in particular \eqref{eq:drag2}).

We are now ready to state a no-contact result. 
\begin{prop}
\label{toygs}
Let $b$ and $\D$ be given in such a way that \eqref{B1}, \eqref{B2}, \eqref{D1}, and \eqref{D2} are satisfied. Then, for every $a$, $\mu$, $h_0$, $\dot h_0$, $\xi_0$, $\dot \xi_0 \in \RR$ with $a$, $\mu$, $h_0 > 0$ there exists a unique global solution to \eqref{gf}, denoted by $(h_{\mu},\xi_{\mu})$. In particular, $h_{\mu}(t) > 0$ for all $t > 0$.
\end{prop}

Having established existence of global solutions, the remainder of the section is dedicated to characterizing the different qualitative behaviors of $h_{\mu}$, as we let $\mu \to 0^+$. For our next result, in addition to the assumptions of \Cref{toygs}, we require that
\begin{enumerate}[label=(B.\arabic*), ref=B.\arabic*]
\setcounter{enumi}{2}
\item \label{B3} $B \ge 0$,
\end{enumerate}
and furthermore, we restrict our attention to the case where the function $\D$ does not depend on the variable $\xi$ and obeys a power law in $h$. To be precise, we assume the following:
\begin{enumerate}[label=(D.\arabic*), ref=D.\arabic*]
\setcounter{enumi}{2}
\item \label{D3} there exist three constants $C_1, C_2 > 0$ and $\alpha \in [1, \infty)$ and a locally Lipschitz continuous function $g \colon (0, \infty) \to [C_1, C_2]$ such that for all $h > 0$ and all $\xi \in \RR$ we have
\[
\D(h,\xi) = g(h)h^{-\alpha}.
\] 
\end{enumerate}
We are now in position to state the first of our main results.

\begin{thm}
\label{mainNB}
Under the assumptions of \Cref{toygs}, set $\xi_0 = \dot \xi_0 = 0$ and let $H \colon [0, \infty) \to [0, \infty)$ be defined via 
\[
H(t) \coloneqq \max \{ 0, h_0 + \dot h_0 t\}.
\]
Then the following statements hold:
\begin{itemize}
\item[$(i)$] Assume that $b(0) = 0$ and that $\dot h_0 < 0$. Then, as $\mu \to 0^+$, we have that $h_{\mu} \to H$ and $\xi_{\mu} \to 0$ uniformly in $[0, t_0]$, where $t_0 \coloneqq - h_0/ \dot h_0$.
\item[$(ii)$] Assume that $b(0) = 0$ and that $\dot h_0 \ge 0$. Then, as $\mu \to 0^+$, we have that $h_{\mu} \to H$ and $\xi_{\mu} \to 0$ uniformly on compact subsets of $[0, \infty)$.
\item[$(iii)$] Assume that $\dot h_0 < 0$, that $b$ satisfies \eqref{B3}, and that $\D$ is given as in \eqref{D3}. Let $\xi \colon [0, \infty) \to \RR$ be defined via $\xi(t) = 0$ if $t \le t_0$, while if $t > t_0$ we let $\xi$ be the unique solution to the initial value problem
\[
\left\{
\arraycolsep=1.4pt\def\arraystretch{1.6}
\begin{array}{l}
\ddot \xi + a b(\xi) = 0, \\
\xi(t_0) = 0,\ \dot \xi(t_0) = - \dot h_0.
\end{array}
\right.
\]
Then, as $\mu \to 0^+$, we have that $h_{\mu}(t) \to H$ and $\xi_{\mu} \to \xi$ uniformly on compact subsets of $[0, \infty)$.
\end{itemize}
\end{thm}

A few comments are in order. First, let us mention that we are primarily interested in the case $\dot h_0 < 0$; the case of a non-negative initial velocity is mainly stated for comparison. Next, observe that \eqref{D3} can be interpreted as a rigidity condition on the solid body (see \Cref{2&3D-LB} and \Cref{2&3D-UB}). It is also worth noting that the function $H$ given in the theorem is monotone. In particular, the conclusions of statement $(iii)$ in \Cref{mainNB} can be summarized as follows: in the vanishing viscosity limit, a ``rigid'' solid (in the sense of condition \eqref{D3}) moving towards the wall will impact the boundary of the container in finite time (to be precise, at $t = t_0$) and it won't separate from the container's wall thereafter. Thus, rather surprisingly, the reduced model predicts that, as we let the viscosity parameter go to zero, the system composed of a smooth rigid shell with an inner mass-spring mechanism (as described in \Cref{red-model-subsec}) approaches a state where the motion of the shell and that of the spring are perfectly decoupled and the shell cannot move away from the wall after collision. This trapping effect is readily explained by observing that if we could instantaneously invert the direction of the velocity, the shell would experience a drag force of equal intensity. More specifically, the resistance of the fluid to the movement of the body does not distinguish on whether the shell is approaching or receding from the wall. Notably, for positive (but small) values of the viscosity parameter, the very same phenomenon that prevents from collision is also the primary obstruction to rebound. \\

Next, we show that the nearly paradoxical situation described by \Cref{mainNB} can be partly resolved by allowing for qualitative changes in the shape of the solid. This has the effect of introducing an asymmetry in the problem which can potentially prevent the trapping phenomenon illustrated above. These changes, however, need to be significant enough to be reflected in the asymptotic behavior of $\D$ (which we recall should be understood as an approximation to the drag force exerted on the body by the surrounding fluid environment) as $h$ approaches zero. 

The running assumptions for the last result of the section are the following:
\begin{enumerate}[label=(B.\arabic*), ref=B.\arabic*]
\setcounter{enumi}{3}
\item \label{B4} $b(y)y > 0$  for all $y \neq 0$;
\end{enumerate}
\begin{enumerate}[label=(D.\arabic*), ref=D.\arabic*]
\setcounter{enumi}{3}
\item \label{D4} $\D$ is non-decreasing as a function of $\xi$, that is,  
\[
\D(h, \xi_1) \le \D(h, \xi_2)
\]
for every $h > 0$ and every $\xi_1 \le \xi_2$;
\item \label{D5} there exist three constants $\delta_1, c_1 > 0$ and $\gamma_1 \in [\alpha,\infty)$ such that for every $h > 0$ we have
\[
\D(h, -\delta_1) \ge c_1h^{-\gamma_1};
\]
\item \label{D6} there exist a constant $\delta_2 > 0$ and a function $\gamma \colon (0, \infty) \to [0, \infty)$ such that
\[
\int_0^h \gamma(y)y^{-1}\,dy \to 0
\] 
as $h \to 0^+$ and with the property that for every $h > 0$ we have
\[
\D(h, - \delta_2) \le \gamma(h)h^{-\gamma_1}.
\]
\end{enumerate}

\begin{thm}
\label{mainB}
Under the assumptions of \Cref{toygs}, let $\xi_0 = \dot \xi_0 = 0$ and assume that $\D$ satisfies \eqref{D4}, \eqref{D5}, and \eqref{D6}. Furthermore, let $b$ be given satisfying \eqref{B4} so that there exists a unique $y^- < 0$ with the property that $2aB(y^-) = \dot h_0^2$. Assume that $y^- < -\delta_2$, where $\delta_2$ is given as in \eqref{D6}. Then, for every subsequence $\{h_n\}_n \subset \{h_{\mu}\}_{\mu}$ there exists $T > t_0 \coloneqq - h_0/ \dot h_0$ such that 
\[
\lim_{n \to \infty} h_{n}(T) > 0. 
\]
\end{thm}

We remark that although \Cref{mainB} holds for every choice of the initial velocity $\dot h_0$, the result is of particular interest in the case where $\dot h_0 < 0$. Indeed, since in this case we have that $h_n(t_0) \to 0$ (see statement $(i)$ in \Cref{mainNB}), the theorem implies that $\{h_n\}_n$ converges to a function which is not monotone.

\begin{rmk}
Notice that \eqref{B4} implies \eqref{B3}. Therefore, the main difference between \Cref{mainNB} and \Cref{mainB} is that condition \eqref{D3} is replaced by \eqref{D4}--\eqref{D6}. We mention here that our prototypical examples for the drag shape factor $\D$ are motivated by the drag force estimates obtained in \Cref{drag-section} $($see in particular \eqref{eq:drag2}$)$ and are given by
\begin{equation}\label{drag_exponent_xi}
d_1(h, \xi) \coloneqq h^{- c \xi - 3/2}, \qquad d_2(h, \xi) \coloneqq h^{-\max\{ \xi, 0\} - 1},
\end{equation}
where $c$ is a positive constant. Notice that $d_1$ and $d_2$ allow for adequate choices of $\delta_2$ and $y^-$. In particular, they satisfy \eqref{D1}, \eqref{D5}, and \eqref{D6}. Notice that the monotonicity requirement in \eqref{D4} holds for all $h \le 1$; thus both examples can be suitably modified to satisfy \eqref{D4}. Additionally, as it becomes apparent from the proof of \Cref{mainB} $($see in particular \eqref{min-est-1}$)$, it is enough to assume that \eqref{D4} holds for $h \le h_0 + \e$, where $\e$ can be any positive number. Finally, notice that \eqref{D2} is automatically satisfies for $d_2$ and holds for $d_1$ provided that $c$ is chosen opportunely.
\end{rmk}

\subsection{Proofs of the main results}
\label{proof-sec}
In this section we collect the proofs of the results stated above. The proofs are inspired by the experimental observations. For \Cref{mainNB} we use the effect that the object is only stopping on a height, where (due to the symmetry of the problem) no escape is possible; we refer to this effect as the ``trapping'' phenomenon. Conversely, in the proof of \Cref{mainB} we exploit the fact that, as the object is approaching the wall, the change of shape effectively stops the particle at a greater distance that the one that would be reached by the undeformed configuration. As the deformation parameter reverts these changes, the symmetry is broken and the elastic response is sufficient to generate a vertical motion away from the wall. 

\begin{proof}[Proof of \Cref{toygs}]
In view of the regularity assumptions (\ref{B1}) and (\ref{D1}), the existence of local solutions to (\ref{gf}) follows directly from Peano's theorem. Let $(h, \xi)$ be a maximal solution defined on the interval $(0,T)$ and assume by contradiction that $T < \infty$. We divide the proof into two steps. 
\newline
\textbf{Step 1:} Multiplying the first equation in $\eqref{gf}$ by $(\dot h - \dot \xi)$, the second one by $a \dot h$, and adding together the resulting expressions, we arrive at
\begin{equation}
\label{energy1}
( \ddot h - \ddot \xi )(\dot h - \dot \xi) + a \ddot h \dot h = - a b(\xi) \dot \xi  - a \mu \D(h, \xi) \dot h^2.
\end{equation}
Define the auxiliary function
\[
F(t) \coloneqq (\dot h(t) - \dot \xi(t))^2 + a \dot h(t)^2 + 2a B(\xi(t))
\]
and notice that integrating (\ref{energy1}) yields
\begin{equation}
\label{EE}
F(t) + 2 a \mu \int_0^t \D(h(s), \xi(s)) \dot h(s)^2 \,ds = F(0).
\end{equation}
Since the integral on the left-hand side is non-negative, in view of (\ref{B2}) we conclude that $\xi, \dot h$, and $\dot \xi$ are bounded. Consequently, since by assumption $T < \infty$, we obtain that $h$ is also bounded in $[0,T]$. This implies that necessarily $h(T) = 0$, since otherwise the solution would admit an extension, hence contradicting the maximality of the solution $(h,\xi)$.
\newline
\textbf{Step 2:} Next, we take $T_1$ to be the smallest time instance for which $h(T_1) = 0$. In view of the previous step we have that $T_1 \le T < \infty$. We notice that by multiplying the second equation in $\eqref{gf}$ by $\chi_{\{\dot h < 0\}}$ we get 
\begin{align*}
\left(\ddot h + b(\xi) \right)\chi_{\{\dot h < 0\}} & = - \mu \D(h,\xi)\dot h \chi_{\{\dot h < 0\}} \\
& \ge - \mu c h^{-\alpha}\dot h \chi_{\{\dot h < 0\}} \\
& \ge - \mu c h^{-\alpha}\dot h \chi_{\{\dot h < 0\}} - \mu c h^{-\alpha}\dot h \chi_{\{\dot h > 0\}} \\
& = - \mu c h^{-\alpha}\dot h,
\end{align*}
where in the first inequality we have used the lower bound given by \eqref{D2}. Integrating both sides in the previous inequality yields 
\begin{equation}
\label{ncpf}
\int_0^t \left(\ddot h(s) + b(\xi(s)) \right)\chi_{\{\dot h < 0\}}(s)\,ds \ge  - \mu c \int_0^t h^{-\alpha}(s)\dot h(s)\,ds = - \mu c \int_{h_0}^{h(t)} y^{-\alpha}\,dy.
\end{equation}
Notice that since by assumption $\alpha \ge 1$, the right-hand side of \eqref{ncpf} tends to infinity as $t \to T_1^-$. Set $U \coloneqq \{s \in (0,t) : \dot h(s) < 0\}$ and observe that if $U = (0,t)$ then  
\begin{equation}
\label{ncpf1}
\int_0^t \ddot h(s) \chi_{\{\dot h < 0\}}(s)\,ds = \dot h(t) - \dot h_0,
\end{equation}
while if this is not the case then we write $U$ as the union of at most countably many disjoint open intervals, i.e.\@
\[
U = \bigcup_{i = 1}^{\infty}(s_i, t_i).
\] 
Without loss of generality we assume that $s_i \le s_j$ if $i \le j$; furthermore, we notice that $\dot h(t_1) = 0$, $\dot h(s_i) = 0$ for all $i \ge 2$, and $\dot h(t_i) = 0$ for all $i \ge 2$ provided that $t_i \neq t$. Consequently, we have 
\begin{equation}
\label{ncpf2}
\int_0^t \ddot h(s) \chi_{\{\dot h < 0\}}(s)\,ds = \sum_{i = 1}^{\infty} \int_{s_i}^{t_i}\ddot h(s)\,ds  = \min\{\dot h(t), 0\} - \min\{\dot h_0, 0\}.
\end{equation}
Combining \eqref{ncpf1} and \eqref{ncpf2} with the bounds obtained in the previous step shows that the left-hand side of \eqref{ncpf} remains bounded as $t \to T_1^-$, thus yielding a contradiction. 

In turn, we obtain that $h > 0$ in $[0, T]$ and the existence of a global solution follows by Step 1. Moreover, the uniqueness of solutions is now direct consequence of the Picard--Lindel\"of theorem and \eqref{D1}.
\end{proof}

In the remainder of this section, we study the asymptotic behavior of solutions as the viscosity parameter $\mu$ approaches zero. To be precise, in the following we fix a sequence $\mu_n \to 0^+$ and denote with $(h_n, \xi_n)$ the solution to \eqref{gf} (given by \Cref{toygs}) relative to the choice $\mu = \mu_n$. 

\begin{lem}
\label{AA}
Under the assumptions of \Cref{toygs}, let $(h_n, \xi_n)$ be solutions as above. Then there exist two Lipschitz continuous functions $h, \xi \colon \RR \to \RR$, with $h$ non-negative, such that (up to the extraction of a subsequence, which we do not relabel) $h_n \to h$ and $\xi_n \to \xi$ uniformly on compact subsets of $[0, \infty)$. 
\end{lem}

\begin{proof}
As a consequence of the energy estimate \eqref{EE}, we see that 
\begin{equation}
\label{unif-bounds}	
\sup \left\{\|\xi_n\|_{L^{\infty}([0,\infty))} + \|\dot h_n\|_{L^{\infty}([0,\infty))} + \|\dot \xi_n\|_{L^{\infty}([0,\infty))} : n \in \NN \right\} < \infty.
\end{equation}

Since the sequence $\{h_n\}_n$ is equi-Lipschitz continuous and $h_n(0) = h_0$ for every $n$, it is also equi-bounded in $[0,T]$ for every $T > 0$. The desired result then follows by the Arzel\`a--Ascoli theorem. 
\end{proof}

\begin{lem}
\label{t<t0lem}
Assume that $b(0) = 0$, $\xi_0 = 0$, $\dot \xi_0 = 0$, and let $(h, \xi)$ be given as in \Cref{AA}. Then, the following hold:
\begin{itemize}
\item[$(i)$] if $\dot h_0 \ge 0$ we have that $h(t) = h_0 + \dot h_0 t$ and $\xi(t) = 0$ for every $t \ge 0$;
\item[$(ii)$] if $\dot h_0 < 0$ we have that $h(t) = h_0 + \dot h_0 t$ and $\xi(t) = 0$ for every $t \le t_0 \coloneqq - h_0/ \dot h_0$.
\end{itemize}
\end{lem}
\begin{proof}
Since by assumption $h(0) = h_0 > 0$, there exists $t_1 > 0$ such that $h(t) > 0$ in $[0, t_1)$. For any $t_2 < t_1$, let $\e \coloneqq \min\{h(t) : t \in [0, t_2]\}$. Then, for $t \in (0, t_2)$ we have 
\[
|\ddot h_n(t)| \le \|b(\xi_n)\|_{L^\infty} + \mu_n \|\dot h_n\|_{L^\infty}\max \left\{\D(y, \xi_n) : y \in [\e, \|h_n\|_{L^\infty}]\right\}.
\]
Thus, $\{h_n\}_n$ is bounded in $C^{1,1}((0, t_2))$ and by the Arzel\`a--Ascoli theorem we find for that for a subsequence $\dot{h}_n\to \dot{h}$ uniformly. 

Next, notice that by integrating the second equation in \eqref{gf} we arrive at
\[
\dot h_n(t) - \dot h_0 = - \int_0^t b(\xi_n(s))\,ds - \mu_n \int_0^t \D(h_n(s), \xi_n(s))\dot h_n(s)\,ds.
\]
Letting $n \to \infty$ in the previous identity yields
\begin{equation}
\label{lim-eq}
\dot h(t) - \dot h_0 = - \int_0^t b(\xi(s))\,ds.
\end{equation}
Subtracting the second equation in \eqref{gf} to the first one we obtain 
\begin{equation}
\label{xin-equation}
\ddot \xi_n = - (1 + a)b(\xi_n) - \mu_n\D(h_n, \xi_n)\dot h_n.
\end{equation}
Therefore, reasoning as above, we conclude that $\xi \in C^1((0, t_2))$ and that eventually extracting a subsequence we also have $\dot \xi_n \to \dot \xi$. Integrating the equation in \eqref{xin-equation} and passing to the limit with respect to $n$ we see that 
\[
\dot \xi(t) = - (1 + a) \int_0^t b(\xi(s))\,ds.
\]
In turn, $\xi$ is of class $C^2$ in $(0, t_2)$ and solves the initial value problem 
\[
\left\{
\arraycolsep=1.4pt\def\arraystretch{1.6}
\begin{array}{l}
\ddot \xi + (1 + a)b(\xi) = 0, \\
\xi(0) = \dot \xi(0) = 0.
\end{array}
\right.
\]
Since by assumption $b(0) = 0$, we readily deduce that $\xi$ is identically equal to zero in $[0, t_2]$. This, together with \eqref{lim-eq}, implies that $\dot h(t) = \dot h_0$ and therefore that $h(t) = h_0 + \dot h_0 t$ for every $t \in [0, t_2]$. 

Finally, assuming first that $\dot h_0 < 0$, we notice that if we can choose $t_1 \ge t_0$ then there is nothing else to do. If this is not the case, then we can assume without loss of generality that $t_1 < t_0$ is such that $h(t_1) = 0$. In this case, letting $t_2 \to t_1^-$ would then imply that $h(t_1) = h_0 + \dot h_0 t_1 > 0$, thus leading to a contradiction. On the other hand, if $\dot h_0 \ge 0$ the proof is similar, but simpler; thus we omit the details. This completes the proof.
\end{proof}

In the following proposition we address the more delicate case in which $\dot h_0 < 0$ and $t \ge t_0$. 

\begin{prop}
\label{mainNB-prop}
Under the assumptions of \Cref{mainNB}, let $h_n$, $\xi_n$, $h$, and $\xi$ be given as in \Cref{AA}. Then, if $\dot h_0 < 0$ we have that $h(t) = 0$ in $[t_0, \infty)$. 
\end{prop}

\begin{proof}
Assume first that $\alpha > 1$ and fix $\e > 0$. We claim that there exists $N(\e) \in \NN$ such that if $n \ge N(\e)$ then 
\begin{equation}
\label{spring0}
h_{n}(t)^{\alpha - 1} \le \e
\end{equation} 
for every $t \ge t_0$. The desired result then follows by the arbitrariness of $\e$. We begin by observing that adding the first equation in \eqref{gf} to a multiple of the second equation yields
\[
(1 + a) \ddot h_n - \ddot \xi_n = - a \mu_n g(h_n)h_n^{-\alpha}\dot h_n.
\]
Let $t \ge t_0$ be such that $h_n(t) < h_0$. Then, integrating the previous identity and by means of a change of variables we obtain 
\begin{align*}
(1 + a)\dot h_n(t) - (1 + a)\dot h_0 - \dot \xi_n(t) & = - a \mu_n \int_0^t g(h_n(s))h_n(s)^{- \alpha}\dot h_n(s)\,ds \\
& = - a \mu_n \int_{h_0}^{h_n(t)}g(y)y^{- \alpha}\,dy \\
& \le - \frac{C_2 a \mu_n}{1 - \alpha}(h_n(t)^{1 - \alpha} - h_0^{1 - \alpha}).
\end{align*}
Rearranging the terms in the previous inequality yields
\[
h_n(t)^{\alpha - 1} \le \frac{C_2 a \mu_n}{\alpha -1 }\left((1 + a)\dot h_n(t) - (1 + a) \dot h_0 - \dot \xi_n(t) + \frac{C_2 a \mu_n}{\alpha - 1} h_0^{1 - \alpha}\right)^{-1}.
\]
Therefore, to prove \eqref{spring0} it is enough to show that
\[
\frac{C_2 a \mu_n}{\alpha - 1} \le \e \left((1 + a)\dot h_n(t) - (1 + a) \dot h_0 - \dot \xi_n(t) + \frac{C_2 a \mu_n}{\alpha - 1} h_0^{1 - \alpha}\right);
\]
in the following it will be convenient to rewrite this condition as
\begin{equation}
\label{wtsNB}
(1 + a)\dot h_0 + \frac{C_2 a \mu_n}{\alpha - 1} \left(\e^{-1} - h_0^{1 - \alpha}\right) \le (1 + a) \dot h_n(t) - \dot \xi_n(t).
\end{equation}
Let us remark here that since by assumption $\dot h_0 < 0$, it is possible to choose $n$ large enough so that the left-hand side in \eqref{wtsNB} is negative. Consequently, if arguing by contradiction we assume that \eqref{wtsNB} does not hold, we obtain that
\begin{equation}
\label{notwtsNB}
\left|(1 + a)\dot h_0 + \frac{C_2 a \mu_n}{\alpha - 1} \left(\e^{-1} - h_0^{1 - \alpha}\right)\right| < |(1 + a) \dot h_n(t) - \dot \xi_n(t)|.
\end{equation}
Squaring both sides in \eqref{notwtsNB} and by Young's inequality we see that 
\begin{align}
\left[(1 + a)\dot h_0 + \frac{C_2 a \mu_n}{\alpha - 1} \left(\e^{-1} - h_0^{1 - \alpha}\right)\right]^2 & < \left( (1 + a) \dot h_n(t) - \dot \xi_n(t)\right)^2 \notag \\
& \le \left(1 + \frac{1}{\delta}\right)(\dot h_n(t) - \dot \xi_n(t))^2 + (1 + \delta)a^2\dot h_n(t)^2 \label{wtsNB1.5} 
\end{align}
holds for every $\delta > 0$. In particular, if we let $\delta = 1/a$, the right-hand side in \eqref{wtsNB1.5} can be rewritten as 
\[
\left(1 + \frac{1}{\delta}\right)(\dot h_n(t) - \dot \xi_n(t))^2 + (1 + \delta)a^2\dot h_n(t)^2 = (1 + a)\left[(\dot h_n(t) - \dot \xi_n(t))^2 + a \dot h_n(t)^2\right],
\]
and therefore, from the energy equality \eqref{EE}, we see that
\begin{multline}
\label{wtsNB2}
(1 + a)\left[(\dot h_n(t) - \dot \xi_n(t))^2 + a \dot h_n(t)^2\right] = (1 + a)^2 \dot h_0^2 - 2 a (1 + a) B(\xi_n(t)) \\ - 2 a (1 + a)\mu_n \int_0^t g(h_n(s))h_n(s)^{-\alpha}\dot h_n(s)^2\,ds.
\end{multline}
Further, expanding the square on the left-hand side of \eqref{wtsNB1.5} we obtain the quantity
\begin{equation}
\label{wtsNB2.5}
(1 + a)^2 \dot h_0^2 + \frac{2 C_2 a (1 + a) \mu_n}{\alpha - 1} \left(\e^{-1} - h_0^{1 - \alpha}\right)\dot h_0 + \mathcal{O}(\mu_n^2).
\end{equation}
Thus, combining \eqref{wtsNB2} and \eqref{wtsNB2.5} with \eqref{wtsNB1.5} for $\delta = 1/a$ and rearranging the terms in the result inequality, we arrive at 
\[
B(\xi_n(t)) + \mu_n \int_0^t g(h_n(s))h_n(s)^{-\alpha}\dot h_n(s)^2\,ds \le \frac{C_2 \mu_n}{\alpha - 1}\left(\e^{-1} - h_0^{1 - \alpha}\right)(-\dot h_0) + \mathcal{O}(\mu_n^2).
\]
Let $t_1 < t_0$ be such that
\begin{equation}
\label{t1contr}
h_0^{1 - \alpha} + \frac{C_2}{C_1}\left(\e^{-1} - h_0^{1 - \alpha}\right) < (h_0 + \dot h_0 t_1)^{1 - \alpha}.
\end{equation}
Then, since by assumption $t \ge t_0$, we have
\begin{align*}
C_1\int_0^{t_1}h_n(s)^{-\alpha}\dot h_n(s)^2\,ds & \le \frac{B(\xi_n(t))}{\mu_n} + \int_0^t g(h_n(s))h_n(s)^{-\alpha}\dot h_n(s)^2\,ds \\
& \le \frac{C_2(-\dot h_0)}{\alpha - 1}\left(\e^{-1} - h_0^{1 - \alpha}\right) + \mathcal{O}(\mu_n).
\end{align*}
We claim that letting $n \to \infty$ in the previous inequality leads to a contradiction to the definition of $t_1$. Indeed, since in $[0,t_1]$ we have that $h_n$ and $\dot h_n$ converge uniformly to $h$ and $\dot h$, respectively. Moreover, since \eqref{B3} implies that $b(0) = 0$, we are in a position to apply \Cref{t<t0lem} and conclude that
\begin{align}
\frac{C_1(- \dot h_0)}{\alpha - 1}\left[(h_0 + \dot h_0 t_1)^{1 - \alpha} - h_0^{1 - \alpha}\right] & = C_1\int_0^{t_1}h(s)^{-\alpha}\dot h(s)^2\,ds \notag \\
& \le \frac{C_2(-\dot h_0)}{\alpha - 1}\left(\e^{-1} - h_0^{1 - \alpha}\right). \label{t1contr2}
\end{align}
As one can readily check, \eqref{t1contr2} is in contradiction with \eqref{t1contr} and the claim is proved. Thus, we have shown that if $t \ge t_0$ and $h_n(t) < h_0$ for every $n$ sufficiently large then $h_n(t)^{\alpha - 1} \le \e$. Assume for the sake of contradiction that there exists $t > t_0$ such that $h_n(t) \to h(t) \ge h_0$. Since $h(t_0) = 0$, there must be a point $\tau \in (t_0, t)$ such that $h(\tau) = h_0/2$. Let $N \in \NN$ be such that 
\[
\left|h_n(\tau) - \frac{h_0}{2}\right| < \frac{h_0}{4}
\]
for all $n \ge N$. Notice that for every such $n$ we have that $h_n(\tau) < h_0$ and therefore $h_n(\tau) \le \e$, provided $n$ is large enough. This implies that $0 < h_0/2 = h(\tau) \le \e$. Letting $\e \to 0$ leads to a contradiction. 

If $\alpha = 1$, the argument presented above can be suitably modified to prove that
\[
\frac{1}{\left| \log h_n(t) - \log h_0\right|} \le \e
\]
rather than \eqref{spring0}. Since the proof requires only minimal changes, we omit the details.
\end{proof}

\begin{lem}
\label{lim-xi}
Under the assumptions of \Cref{toygs}, let $h_n$, $\xi_n$, $h$, and $\xi$ be given as in \Cref{AA}. Then, if $\dot h_0 < 0$, $b(0) = 0$, and $h(t) = 0$ for $t \ge t_0$ we have that $\xi$ is the unique solution to the initial value problem 
\[
\left\{
\arraycolsep=1.4pt\def\arraystretch{1.6}
\begin{array}{l}
\ddot \xi + a b(\xi) = 0, \\
\xi(t_0) = 0,\ \dot \xi(t_0) = - \dot h_0.
\end{array}
\right.
\]
\end{lem}
\begin{proof}
We begin by noticing that if we let $z_n \coloneqq h_n - \xi_n$, we have that $\ddot z_n = a b(\xi_n)$. Therefore, the sequence $\{z_n\}_n$ is bounded in $C^{1,1}$. In turn, by Arzel\'{a}-Ascoli, we see that there exists a function $z$ such that, up to the extraction of a subsequence (which we do not relabel), $z_n \to z$ and $\dot z_n \to \dot z$ uniformly on compact subsets of $[0, \infty)$. Integrating the equation for $z_n$ and passing to the limit in $n$ yields
\[
\dot z(t) - \dot z(0) = \int_0^t a b(\xi(s))\,ds.
\]
Therefore $z \in C^2(0,\infty)$ and satisfies $\ddot z = a b(\xi)$. Since by assumption we have that $h(t) = 0$ for every $t \ge t_0$, in view of \Cref{t<t0lem} we are left to show that $\dot \xi(t_0) = -\dot h_0$. This follows by observing that 
\[
\dot h_0 = \lim_{t \nearrow t_0} \dot h(t) - \dot \xi(t) = \dot z(t_0) = \lim_{t \searrow t_0} \dot z(t) = \lim_{t \searrow t_0} - \dot \xi(t).
\]
This concludes the proof.
\end{proof}

\begin{proof}[Proof of \Cref{mainNB}] Combining the results of \Cref{t<t0lem}, \Cref{mainNB-prop} and \Cref{lim-xi}, we obtain that for every sequence $\mu_n \to 0^+$, the corresponding sequences of solutions, i.e.\@ $\{h_n\}_n$ and $\{\xi_n\}_n$, admit a subsequence with the desired convergence properties. As one can readily check with a standard argument by contradiction, this implies that the convergence holds for the entire family. Hence, the proof is complete.
\end{proof}

We conclude the section with the proof of \Cref{mainB}.

\begin{proof}[Proof of \Cref{mainB}]
The proof is for the case that $\gamma_1 > 1$ (in (\eqref{D5})). The case that $1 = \alpha = \gamma_1$ follows by the same arguments (replacing powers by logarithms at the relevant places).

We divide the proof into several steps.
\newline
\textbf{Step 1:} 
Arguing by contradiction, assume that $h_n(t) \to \max\{h_0 + \dot h_0 t , 0\}$ for all $t \ge 0$. Then, an application of \Cref{lim-xi} yields that eventually extracting a subsequence we have that $\xi_n \to \xi$, where $\xi$ is the solution to
\begin{equation}
\label{non-lin-osc}
\ddot \xi + a b(\xi) = 0,
\end{equation}
with initial conditions $\xi(t_0) = 0$ and $\dot \xi(t_0) = - \dot h_0$. Furthermore, from \eqref{B2} and \eqref{B4} we see that there are exactly two points $y^-, y^+$, with $y^- <0< y^+$, such that $2a B(y^{\pm}) = \dot h_0^2$. Let 
\[
t^{\pm} \coloneqq 2 \left|\int_0^{y^{\pm}} \left(\dot h_0^2 - 2aB(y)\right)^{-1/2}\,dy\right|.
\]
Observe that $t^\pm$ are finite by the positivity assumption of (B.4), since (by Taylor expansion) $\dot h_0^2 - 2aB(y)=2ab(y^\pm)(y^\pm-y)+\mathcal{O}((y^\pm-y)^2)$.
Further notice that the points $y^{\pm}$ are turning points for the non-linear oscillator \eqref{non-lin-osc}, whose period is given by $t^+ + t^-$. We then define $t_1 \coloneqq t_0 + t^+$ and $t_2 \coloneqq t_0 + t^+ + t^-$. With this notation at hand, we have that 
\[
\arraycolsep=1.4pt\def\arraystretch{1.6}
\begin{array}{rll}
y^- \le \xi(t) \le & y^+\qquad  & \text{ for } t \ge 0, \\
\xi(t) > & 0 & \text{ for } t \in (t_0, t_1), \\
\xi(t) < & 0 & \text{ for } t \in (t_1, t_2).
\end{array}
\]
\newline
\textbf{Step 2}: In this step we prove that for every $\e > 0$ with $6\e < t_2 - t_1$ there exists $N(\e)$ such that if $n \ge N(\e)$ then $\dot h_n(t) \ge 0$ in $(t_1 + 3\e, t_2 -3 \e)$. To this end, observe that by the uniform convergence of $\xi_n$ to $\xi$, there exists a positive $\delta$ such that 
\begin{equation}
\label{xi-negative}
\xi_n(t) \le - \delta	
\end{equation}
for all $t \in (t_1 + \e, t_2 - \e)$ and all $n$ sufficiently large. Arguing by contradiction, suppose that for a subsequence of $\{h_n\}_n$ (which we do not relabel) we can find  points $\tau_n \in (t_1 + 3\e, t_2 - 3\e)$ with the property that $\dot h_n(\tau_n) < 0$. Observe that necessarily $\dot h_n(t) \le 0$ in $(t_1 + \e, \tau_n)$. Indeed, if this was not the case then $h_n$ would admit a local maximum at a point $\sigma_n$ in this interval. This leads to a contradiction since \eqref{xi-negative}, together with \eqref{B4}, implies that $\ddot h_n(\sigma_n) = - b(\xi_n(\sigma_n)) > 0$.
Let $t_{\e, n} \in (t_1 + \e, t_1 + 2 \e)$ be such that
\[
h_n(t_1 + \e) - h_n(t_1 + 2\e) = \dot h_n(t_{\e, n})\e.
\]
Letting $n \to \infty$ we see that $\dot h_n(t_{\e, n}) \to 0$. Integrating the second equation in \eqref{gf} between $t_{\e, n}$ and $t \in (t_{\e, n}, t_1 + 3\e)$, using the fact that $h_n$ is non-increasing in this interval, and \eqref{xi-negative} we arrive at
\[
\dot h_n(t) - \dot h_n(t_{\e, n}) \ge - \int_{t_{\e, n}}^t b(\xi_n(s))\,ds \ge \beta(t - t_{\e, n}),  
\]
where $\beta \coloneqq \min\{-b(y) : y \in [y^-, - \delta]\}$. Integrating the previous inequality between $t_1 + 2\e$ and $t_1 + 3\e$ we then conclude that
\[
h_n(t_1 + 3\e) - h_n(t_1 + 2\e) - \dot h_n(t_{\e, n})\e \ge \frac{\beta \e^2}{2},
\]
which in turn implies
\[
h_n(t_1 + 3\e) - h_n(t_1 + \e)\geq \frac{\beta \e^2}{2}.
\]
Letting $n \to \infty$ leads to a contradiction, since the left hand side was assumed to converge to zero.
\newline  
\textbf{Step 3:} Let $t_n$ be such that
\[
h_n(t_n) = \min\left\{h_n(t) : t \in \left[0, \frac{t_1 + t_2}{2}\right]\right\}.
\]
The purpose of this step is to prove that there exists a constant $K > 0$ such that for every $n$ sufficiently large we have
\begin{equation}
\label{mu-lowerbound}
K h_n(t_n)^{\gamma_1 - 1} \ge \mu_n.
\end{equation}
To see this, fix $\e > 0$ such that $\xi(t) > - \delta_1/2$ in $[0, t_1 + 3\e]$, where $\delta_1$ is given as in \eqref{D5}. Using the fact that $\xi_n \to \xi$ uniformly in $[0, t_2]$ and the result of the previous step it is possible to find a number $N(\e)$ such that if $n \ge N(\e)$ then the following properties are satisfied:
\begin{equation}
\label{tn-min}
\arraycolsep=1.4pt\def\arraystretch{1.6}
\begin{array}{rll}
\xi_n(t) \ge & -\delta_1 & \text{ in } [0, t_1 + 3\e], \\
\dot h_n(t) \ge & 0 & \text{ in } [t_1 + 3\e, t_2 - 3\e].
\end{array}
\end{equation}
Notice that by \eqref{tn-min} it follows that $t_n \in [0, t_1 + 3\e]$. Moreover, in view of \eqref{D4}, \eqref{D5}, and \eqref{tn-min}, for every $t \in [0, t_1 + 3\e]$ we have
\begin{align}
\label{min-est-1}
\left( \ddot h_n(t) + b(\xi_n(t)) \right)\chi_{\{\dot h_n \le 0\}}(t) & = - \mu_n \D(h_n(t), \xi_n(t)) \dot h_n(t) \chi_{\{\dot h_n \le 0\}}(t) \notag \\
& \ge - \mu_n \D(h_n(t), - \delta_1) \dot h_n(t) \chi_{\{\dot h_n \le 0\}}(t) \notag \\
& \ge - \mu_n c_1 h_n(t)^{- \gamma_1}\dot h_n(t) \chi_{\{\dot h_n \le 0\}}(t) \notag \\
& \ge - \mu_n c_1 h_n(t)^{- \gamma_1}\dot h_n(t).
\end{align}
Reasoning as in \eqref{ncpf2} (see also \eqref{unif-bounds}), we conclude that there exists a constant $k > 0$ such that
\begin{equation}
\label{k-mu}	
k \ge \int_0^{t_n} \left( \ddot h_n(t) + b(\xi_n(t)) \right)\chi_{\{\dot h_n \le 0\}}(t)\,dt \ge \frac{\mu_n c_1}{\gamma_1 - 1} \left(h_n(t_n)^{1 - \gamma_1} - h_0^{1 - \gamma_1}\right),
\end{equation}
where the second inequality is obtained by integrating the estimate in \eqref{min-est-1}. Notice that \eqref{k-mu} can be rewritten as
\begin{equation*}
\frac{(\gamma_1 - 1)k}{c_1} h_n(t_n)^{\gamma_1 - 1} \ge \mu_n \left(1 - h_0^{1 - \gamma_1}h_n(t_n)^{\gamma_1 - 1}\right),
\end{equation*}
and that the right-hand side can be further estimated from below by $\mu_n/2$, provided $n$ is large enough. In particular, we have show that \eqref{mu-lowerbound} holds for $K = 2(\gamma_1 - 1)k/c_1$.
\newline
\textbf{Step 4:} With this estimate at hand can proceed as follows. Since by assumption $y^- < - \delta_2$, where $\delta_2$ is the constant given as in \eqref{D6}, eventually replacing $\e$ with a smaller number, we can find $T_1, T_2$ such that $T_1 < T_2$, $4\e < T_2 - T_1$, $(T_1, T_2) \subset (t_1 + 3\e, t_2 - 3\e)$, and with the property that $\xi_n(t) \le - \delta_2$ for every $t \in (T_1, T_2)$. Reasoning as in Step 2 of the proof, for every $n$ we can find a point $\tau_{\e, n} \in (T_1, T_1 + \e)$ in such a way that 
\[
h_n(T_1) - h_n(T_1 + \e) = \dot h_n(\tau_{\e, n})\e.
\]
Notice that for every $t \in (T_1, T_2)$, \eqref{mu-lowerbound} and \eqref{D5} imply that
\begin{align*}
\ddot h_n(t) & = - b(\xi_n(t)) - \mu_n \D(h_n(t), \xi_n(t))\dot h_n(t) \\
& \ge - b(\xi_n(t)) - K h_n(t_n)^{\gamma_1 - 1}\gamma(h_n(t))h_n(t)^{- \gamma_1}\dot h_n(t) \\
& \ge - b(\xi_n(t)) - K \gamma(h_n(t))h_n(t)^{-1}\dot h_n(t).
\end{align*}
Integrating the previous inequality between $\tau_{\e, n}$ and $t \in (\tau_{\e, n}, T_2)$ we obtain 
\begin{align}
\label{bounce-dot}
\dot h_n(t) - \dot h_n(\tau_{\e, n}) & \ge - \int_{\tau_{\e, n}}^t b(\xi_n(s))\,ds - K \int_{\tau_{\e, n}}^t \gamma(h_n(s))h_n(s)^{-1}\dot h_n(s)\,ds \notag \\
& = - \int_{\tau_{\e, n}}^t b(\xi_n(s))\,ds - K \int_{h_n(\tau_{\e, n})}^{h_n(t)}\gamma(y)y^{-1}\,dy \notag \\
& \ge - \int_{\tau_{\e,n}}^t b(\xi_n(s))\,ds - K \int_0^{h_n(T_2)}\gamma(y)y^{-1}\,dy, 
\end{align}
where in the last inequality we have used the fact that $h_n$ is non-decreasing in $(T_1, T_2)$. Integrating \eqref{bounce-dot} from $\tau_{\e, n}$ to $T_2$ yields
\begin{align*}
h_n(T_2) - h_n(\tau_{\e, n}) - \dot h_n(\tau_{\e, n})(T_2 - \tau_{\e, n}) &\ge - \int_{\tau_{\e, n}}^{T_2} \int_{\tau_{\e,n}}^t b(\xi_n(s))\,dsdt \\
&\quad - K(T_2 - \tau_{\e,n}) \int_0^{h_n(T_2)}\gamma(y)y^{-1}\,dy.
\end{align*}
In view of \eqref{D6}, by letting $n \to \infty$ in the previous inequality we obtain
\begin{equation}
\label{rhs-contr->}
0 = \lim_{n \to \infty} h_n(T_2) - h_n(\tau_{\e, n}) - \dot h_n(\tau_{\e, n})(T_2 - \tau_{\e, n}) \ge \lim_{n \to \infty} - \int_{\tau_{\e, n}}^{T_2} \int_{\tau_{\e,n}}^t b(\xi_n(s))\,dsdt.
\end{equation}
To conclude, it is enough to notice that the right-hand side of \eqref{rhs-contr->} is positive. Indeed, if we set $\tilde{\beta} \coloneqq \min \{ -b(y) : y \in [y^-, - \delta_2] \} > 0$, we get
\[
- \int_{\tau_{\e, n}}^{T_2} \int_{\tau_{\e,n}}^t b(\xi_n(s))\,dsdt \ge \frac{1}{2}(T_2 - T_1 - \e)^2 \tilde{\beta} > 0.
\]
We have thus arrived at a contradiction and the proof is complete.
\end{proof}

\section{Numerical results}
\label{sec:num}
In this section, we present some numerical experiments in order to further strengthen our main conjecture. We begin by illustrating that the ``effectively deformable'' reduced model, for which the internal spring deformation is coupled with the damping term that represents the drag force, does indeed produce a physical rebound. We conclude the section with the comparison from a numerical standpoint of the ODE and PDE solutions. The striking similarities that we observe suggest the relevance of the reduced model for the description of the rebound phenomenon.

\subsection{Reduced model}\label{sec:ODE_num}
In the numerical simulations we shall consider a particular variant of the reduced model \eqref{gf}. To be precise, we take 
\[
\D(h,\xi) \coloneqq \frac{c_1 h^{-c_2\xi - 3/2} +  c_3}{M},
\]
where the first term on the right-hand side is in accordance with \eqref{drag_exponent_xi} 
and reflects the change of flatness parameterized by $\xi$, and the second constant term describes the standard Stokes drag in the absence of geometrical constraints. Furthermore, let $a$ and $b$ be given as in \eqref{choices-red}, then the system of governing equations for $h$ and $\xi$ can be written as
\begin{align*}
M \ddot h &= - k\xi - \mu\left(c_1 h^{-c_2\xi - 3/2} +  c_3\right)\dot{h},\\
m( \ddot h - \ddot \xi) & = k\xi,
\end{align*}
with initial conditions  
\begin{align*}
h(0) &= h_0,\hspace{1cm}\dot{h}(0)= \dot{h}_0,\\
\xi(0) &= \xi_0,\hspace{1cm}\dot{\xi}(0)= \dot{\xi}_0.
\end{align*}

\subsubsection{Numerical results}\label{sec:ODE_results}
In order to demonstrate the critical effect of the change of flatness for the reduced model, we compare the two situations in which $c_2 = 0$ and $c_2 \neq 0$. In both cases an internal energy storage mechanism is present in the form of a mass-spring element. In the first case the elongation of the spring does not affect the drag force (rigid shell model). For this model we prove that a physical rebound is not possible, see \Cref{toygs} and \Cref{mainNB}. In the other case, the elongation of the spring does affect the drag force (effectively deformable model); in this setting a physical rebound can be expected in view of \Cref{mainB}.

The qualitative different behaviors that the two settings can exhibit are summarized in \Cref{fig:giovanni}, where we plot the evolution of $h$, that is, the distance to the wall, as a function of time $t$ for several values of the fluid viscosity $\mu$ for the rigid shell model (left column) and for the effectively deformable model (right column). The top row shows a larger time interval $(0,2)$\,s, while on the bottom row we zoom into the vicinity of the supposed rebound instant. The figure clearly demonstrates the critical effect of the inclusion of a coupling between the internal deformation parameter $\xi$ and the drag force on the dynamics of the system. For the rigid shell model, the response converges with decreasing viscosity to the ``hit-and-stick'' solution, i.e., to the piecewise affine function $H(t) = \max \{0, h_0 + \dot{h}_0 t\}$ (see \Cref{mainNB}). Note that as a result of the presence of the internal spring, the solutions for the rigid shell model are non-monotone, but as the amplitude of these oscillations diminishes with decreasing fluid viscosity, this bouncing does not correspond to the physical rebound as we defined it (that is, it doesn't withstand the vanishing viscosity limit).

\begin{figure}[!htbp]
\begin{center}
\includegraphics[width=7.3cm]{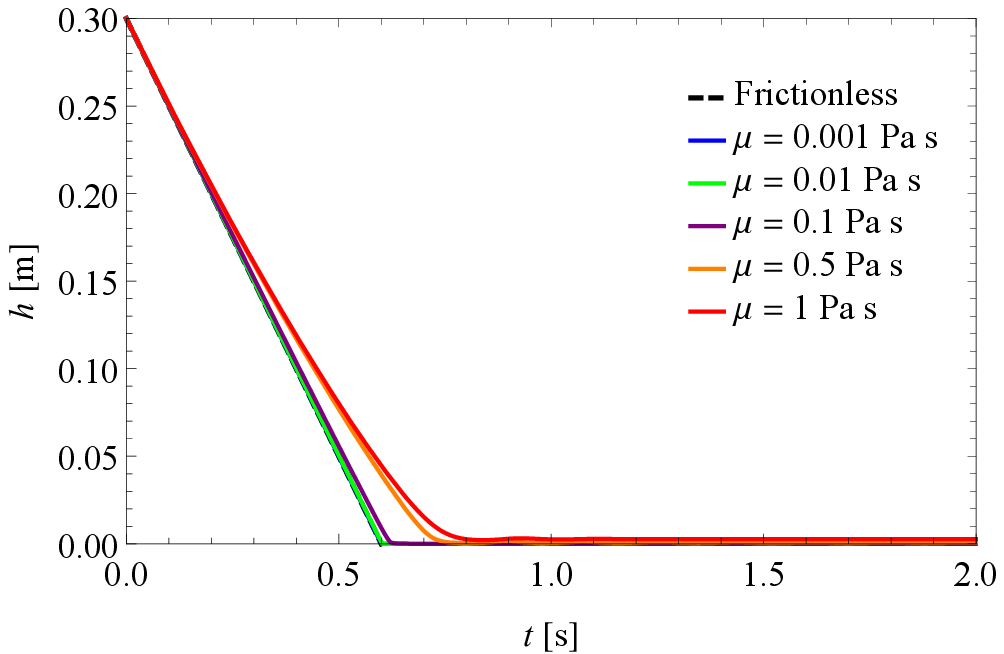}\hfill
\includegraphics[width=7.3cm]{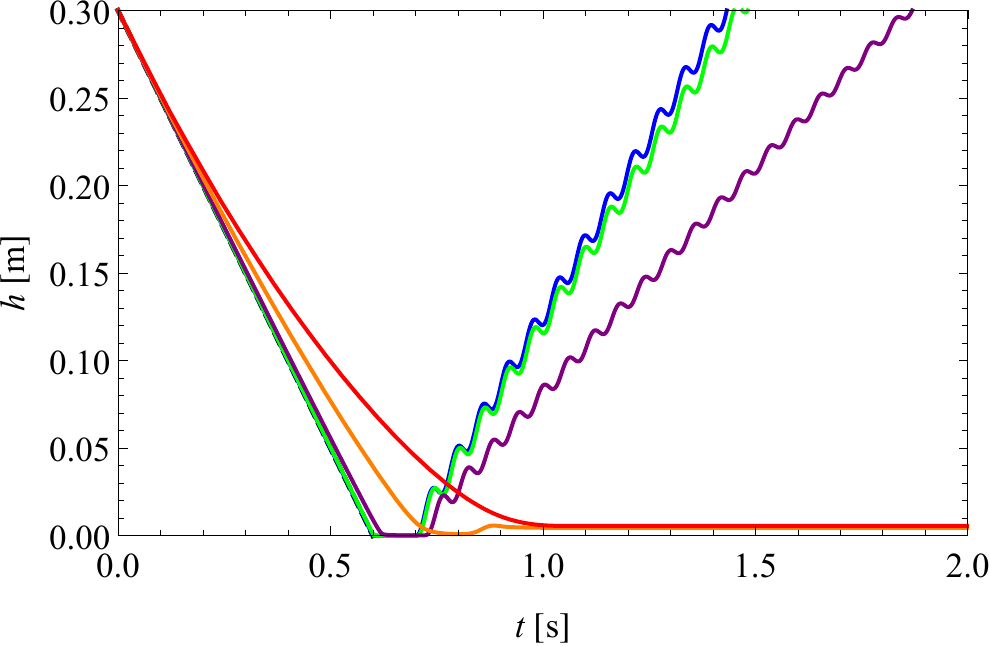}\\[10pt]
\includegraphics[width=7.15cm]{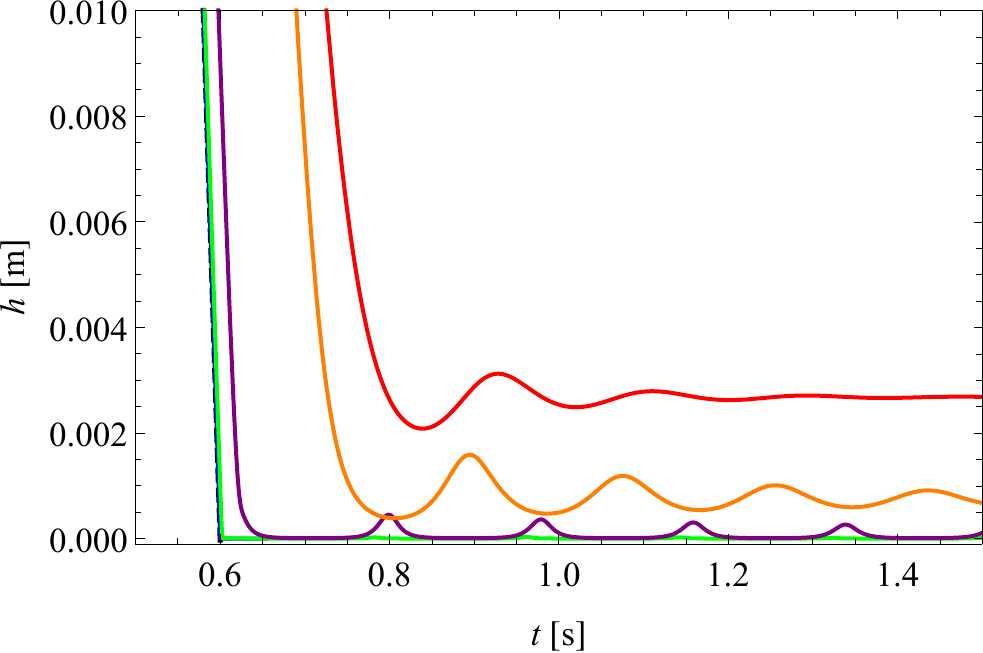}\hfill
\includegraphics[width=7.15cm]{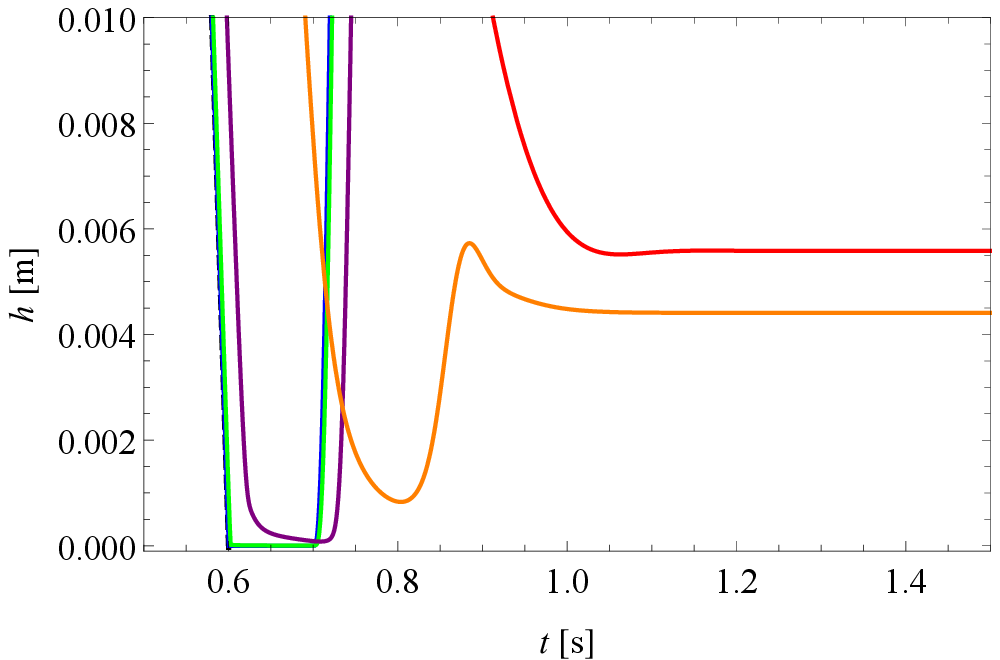}\hspace*{2mm}\ \\
\caption{Physical rebound is not possible for the rigid shell model (left). The effectively deformable model can produce a rebound (right). The graphs in the bottom row are close-ups in the vicinity of the supposed rebound instant.} \label{fig:giovanni}
\end{center}
\end{figure}

The situation is very different for the effectively deformable model. For the highest values of viscosity (red line), the body bounces off very mildly and its motion is rather quickly slowed down due to friction in the fluid. But with decreasing viscosity, the rebound is getting more and more pronounced and the solutions appear to be converging to an expected frictionless limit. Note how oscillations of the internal spring manifest themselves in the motion of $h$, becoming less and less damped as $\mu$ goes to zero. Interestingly, our simulations indicate that the kinetic energy corresponding to the outer shell, i.e.\@ the fraction $M/(M + m)$ of the total kinetic energy of the system, is lost during the rebound in the vanishing viscosity limit. This suggests that a proper physical rebound (i.e.\@ a perfectly elastic vacuum situation), would correspond in our reduced model to the case $M \to 0^+$, that is, to the situation in which the entire mass of the body is carried by the internal mass and the outer shell is massless.

The values of the parameters used in the depicted simulations are as follows: 
$k = 10000$, $c_1 = 0.1$, $c_2=20$ for the effectively deformable model (and it is set equal to zero in the rigid case), $c_3=7.4$, $M = 1$, and finally $m = 8.2$. \footnote{Please observe that these parameters are in accordance with the assumptions of \Cref{mainB}. Indeed, since the energy estimate~\eqref{EE} implies that $\sup_t \abs{\xi(t)} \le \frac{\dot{h}_0}{\sqrt{k}} = \frac{1}{200}$, we have that $\xi(t) < \frac{1}{2c_2} = \frac{1}{40}$ for all $t > 0$.}
This particular choice of the parameters is motivated by our effort to match the solutions to the reduced model with the finite element solutions to the full FSI problem described in \Cref{sec:ODE-FEM_comparison}. For the initial conditions we considered the following values: $h_0 = 0.3$, $\dot h_0 = 0.5$, and $\xi_0 = \dot \xi_0 = 0$.

\subsection{Full FSI model}
The standard form of the fluid-structure interaction problem, as given in \Cref{section:full-FSI}, consists of two sets of equations --- one for the fluid and one for the solid --- which are formulated in different configurations. While the fluid component is described in the physical Eulerian configuration, the equations for the solid are formulated in the reference (Lagrangian) configuration. Usually, the fluid-structure interaction is treated by the so-called arbitrary Lagrangian-Eulerian (ALE) method (see for example \cite{Donea2004,Scovazzi2007}) where the solid part is Lagrangian, but the fluid problem is transformed into a certain special configuration which reflects the changes of the shape of the fluid domain but is not disrupted by the (possibly vigorous) motion of the fluid within the domain. The ALE method can be used to tackle the problem of contact in fluid-structure interaction, but often requires the use of sophisticated adaptive remeshing techniques to keep the fluid domain in the contact region well resolved. For our specific problem, we find it more convenient to solve the whole problem in the Eulerian setting, where the interaction conditions \eqref{interaction_conditions} are satisfied automatically. In order to do so, the problem for the solid must be first transformed accordingly. Thus, we rewrite the Eulerian form of the momentum balance for the solid as
\[
\rho_s \left(\pder{\bm{v}}{t} + \bm{v} \cdot \nabla \bm{v} \right) = \di \sigma_s,
\]
where the Cauchy stress $\sigma_s$ is given by \eqref{stress-NeoHooke}, i.e.,
\[
\sigma_s = -p\I_N + G \left(\Bb - \frac{1}{N}(\tr \Bb)\I_N \right).
\]
The evolution equation for the Cauchy--Green tensor $\Bb$ can be derived directly from the kinematics. Indeed, the material time derivative of the deformation gradient, denoted by $\dot{\F}$, can be computed as follows:
\begin{equation}
\label{fdot-1}
\dot{\F} = \frac{\partial\F({\bm X},t)}{\partial t}=\frac{\rm d}{{\rm d}t}\F(\bm{x}(\bm{X},t),t)= \pder{\F({\bm x},t)}{t} + (\bm{v} \cdot \nabla)\F({\bm x},t),
\end{equation}
where in the second equality, we switch from the Lagrangian to the Eulerian description. On the other hand, by directly employing the definition of $\F$, we get
\begin{align}
\dot{\F} &=\frac{\partial \F({\bm X},t)}{\partial t} = \frac{\partial}{\partial t}(\nabla_{\bm X}\eta({\bm X},t)) = \nabla_{\bm X}\frac{\partial\eta({\bm X},t)}{\partial t} =  \nabla_{\bm X}\bm{v}({\bm X},t)=\nabla_{\bm X}(\bm{v}({\bm x}({\bm X},t),t) \notag \\
& = \pder{\bm{v}({\bm x},t)}{{\bm x}}\pder{\bm {x}({\bm X},t)}{{\bm X}}=\pder{\bm{v}({\bm x},t)}{{\bm x}}\F({\bm X},t)=(\nabla \bm{v}({\bm x},t))\F(\bm{x}, t), \label{fdot-2}
\end{align}    
where $\nabla_{\bm X}$ denotes the gradient with respect to the reference position ${\bm X}$ (note that in the fifth and last equality, we again switch from the Lagrangian to the Eulerian description). Combining \eqref{fdot-1} and \eqref{fdot-2} gives a classical kinematic relation between the deformation gradient and the velocity in the Eulerian setting. Furthermore, we obtain
\[
\dot{\Bb}({\bm x},t)=\dot{\F}\F^{\rm T}+\F\dot{\F}^{\rm T}=(\nabla\bm{v})({\bm x}, t)\Bb({\bm x},t)+\Bb({\bm x},t)(\nabla\bm{v}({\bm x},t))^{\rm T},
\]
which enables to close the system of equations. To summarize, the governing equations for the incompressible neo-Hookean solid in the Eulerian setting read
\begin{align*}
\di \bm{v} &=0,\\
\rho_s \left(\pder{\bm{v}}{t} + \bm{v} \cdot \nabla \bm{v} \right) &= \di \sigma_s,\quad \sigma_s = -p\I_N + G\Bb^d,\\
\pder{\Bb}{t}+(\bm{v} \cdot \nabla)\Bb - (\nabla \bm{v}) \Bb - \Bb (\nabla \bm{v})^{\rm T} &= \mathbb{O}.
\end{align*}
Here and in the following $\mathbb{O}$ denotes the zero matrix. Now, since both the fluid and the solid are  described in the Eulerian frame of reference, we distinguish between the two simply by rheology. The formula for the Cauchy stress can be written in a unifying (essentially visco-elastic) manner as 
\[
\sigma \coloneqq -p\I_N + 2 \mu \mathbb{D}(\bm{v}) + G\Bb^d,
\]
where 
\[
\arraycolsep=3.4pt\def\arraystretch{2.2}
\begin{array}{ll}
\displaystyle \mu = \mu_f > 0, \ G = G_f = 0, \ \Bb = \I_N & \text{ in } \mathcal{F}(t), \\
\displaystyle \mu = \mu_s = 0, \ G = G_s > 0, \ \pder{\Bb}{t} + (\bm{v} \cdot \nabla) \Bb - (\nabla \bm{v}) \Bb - \Bb(\nabla\bm{v})^{\rm T} = \mathbb{O} & \text{ in } \B(t).
\end{array}
\]
Thus, the Cauchy stress $\sigma$ is equal to $\sigma_f$ in the fluid and to $\sigma_s$ in the solid. For the numerical implementation of the above model, we employ the conservative level-set method with reinitialization, which facilitates the tracking of the boundary between the fluid and the solid domain. In particular, we add a new scalar unknown $\chi$, defined via
\[
\chi(\bm{x}, t) \coloneqq \left\{
\begin{matrix}
1& \text{ if } \bm{x} \in \B(t),\\
0& \text{ if } \bm{x} \in \mathcal{F}(t),
\end{matrix}
\right.
\]
and which is smeared out so that it changes smoothly across an interfacial zone with  characteristic thickness $\e$. As the elastic solid moves in the fluid, the level set function $\chi$ is advected by the fluid velocity, and, in order to ensure stability of the method and a good resolution of the interfacial zone, the level set must also be reinitialized during the simulations (see for example \cite{Olsson2005}). The equations for these two processes read as follows:
\begin{equation*}\label{eq_levelset}
\frac{\partial \chi}{\partial t}+\bm{v}\cdot\nabla\chi=0,\qquad
\di\left[\chi(1-\chi)\frac{\nabla\chi}{|\nabla\chi|}\right]-\e\Delta\chi=0.
\end{equation*}
Note that the reinitialization smears out the level set function to the required diffuse profile. In one dimension, this reads 
\begin{equation}\label{diffuse_interface_tanh}
\chi = \frac12\left(1+\tanh\frac{x}{2\e}\right),
\end{equation}
so indeed the parameter $\e$ controls the thickness of the diffuse interface. Similarly, the material parameters $\rho, \mu$, and $G$ are prescribed to change smoothly across the interface by setting
\[
\rho = \chi\rho_s +(1-\chi)\rho_f,\quad \mu = \chi\mu_s+(1-\chi)\mu_f,\quad G = \chi G_s+(1-\chi)G_f.
\]
In order to reduce the complexity of the problem (enhanced by the necessity to solve for the evolution of the tensor $\Bb$), we simplified the model by assuming that the elastic deformations and the velocities are small (a valid assumption in the considered applications). Consequently, we omit the convective terms in the evolution equation for $\Bb$, and, in the same spirit, we assumed that $(\nabla \bm{v})\Bb \sim \nabla\bm{v}$ in the solid, while in the fluid the equation is regularized in such a way that the evolution equation for $\Bb$ can be solved. Thus, the model we solved numerically is the following
\begin{align*}
    \di \bm{v}&=0,\\
    \pder{\chi}{t} + \bm{v} \cdot \nabla \chi &=0,\\
    \rho \pder{\bm{v}}{t} &= \di \sigma,\quad \sigma = -p\I_N + 2 \mu \mathbb{D} + G \mathbb{B}^d,\\
    \pder{\Bb_s}{t} & = 2 \mathbb{D},\quad \pder{\Bb_f}{t} + \Bb_f - \I_N = \mathbb{O},\quad \Bb = (1-\chi)\Bb_f + \chi\Bb_s,
\end{align*}
equipped with the reinitialization to produce the smooth interface of $\xi$.
The problem is implemented with the finite element method in the open source finite element library FEniCS (\cite{Alnaes2015}) and discretized on a regular triangular mesh. Finally, the equation for $\Bb_s$ is solved locally and $\Bb_s$ is then inserted immediately into the balance equation of linear momentum. In the numerical implementation, the time derivatives are approximated with the backward Euler time scheme, that is,
\[
\pder{\Bb}{t} \sim \frac{\Bb - \Bb_0}{\Delta t},
\]
where $\Bb_0$ is the value of $\Bb$ at the previous time step and $\Delta t$ is a time step that is chosen adaptively according to the speed of the fluid from the previous time step in such a way that the CFL condition holds, that is
\[
\frac{\Delta t\, v_{\max}}{h_{\min}} = \frac12.
\]
Here $v_{\max}$ denotes the maximum value of the velocity magnitude and $h_{\min}$ is the minimum size of the element. The local integration of $\Bb$ gives
\begin{align*}
    \Bb_s & = \Bb_0 + 2 \mathbb{D} \Delta t,\\
    \Bb_f & = \frac{\I_N + \Bb_0 \Delta t}{1 + \Delta t},\\
    \Bb & = (1 - \chi) \Bb_f + \chi \Bb_s.
\end{align*}
Thus, the only global unknowns are the velocity $\bm{v}$, the pressure $p$ and the level-set function $\chi$. While velocity and pressure are approximated by the classical Taylor--Hood element P2/P1, the level-set function is approximated with the P2 element. The non-linearities are treated with the exact Newton method and the resulting set of linear equations is then solved with the direct solver MUMPS.

\subsubsection{Numerical results}\label{sec:FEM_results}
We have numerically investigated the rebound of an elastic ball in a viscous fluid environment. The radius of the ball considered is 0.2\,m and its center is initially located 0.5\,m from the bottom wall in a square container of size 0.8\,m. At the boundary of the container, no-slip boundary conditions are prescribed. The initial velocity of the ball is 0.5\,m/s (downwards) and the fluid is initially at rest. Throughout the simulation, body forces have been switched off. The following material parameters have been prescribed:
\begin{align*}
    \rho_f &=1.0\ \text{kg/m}^3,&\hspace*{-0cm} \rho_s &=1001.0\ \text{kg/m}^3,\\
    \mu_f &=0.1\  \text{Pa s},&\hspace*{-0cm} \mu_s &=0.0\ \text{Pa s},\\
    G_f &=0.0\ \text{Pa},&\hspace*{-0cm} G_s &=50\,000.0\ \text{Pa}.
\end{align*}

\begin{figure}[!h]
\begin{center}
\includegraphics[width=5cm]{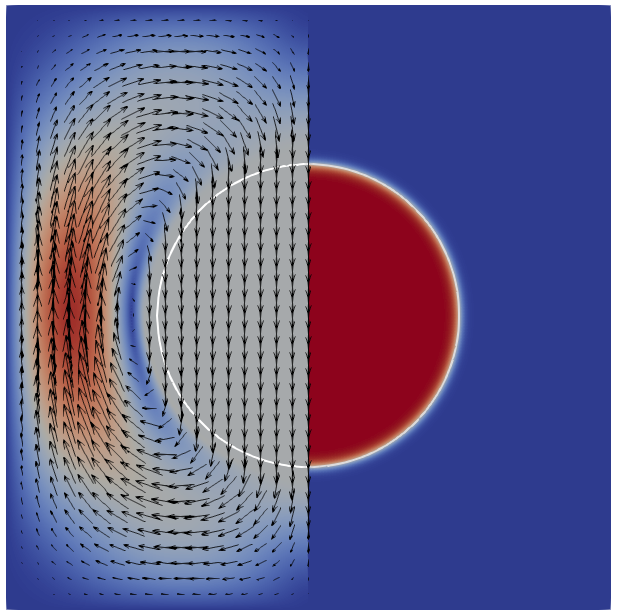}\hfill
\includegraphics[width=5cm]{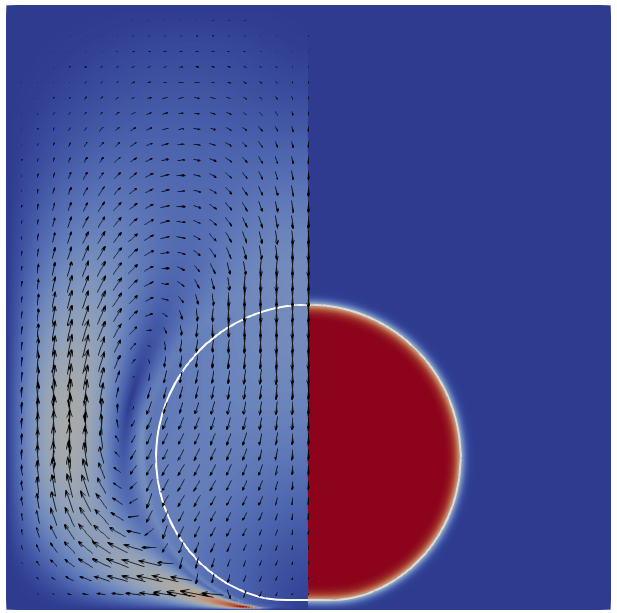}\hfill
\includegraphics[width=5cm]{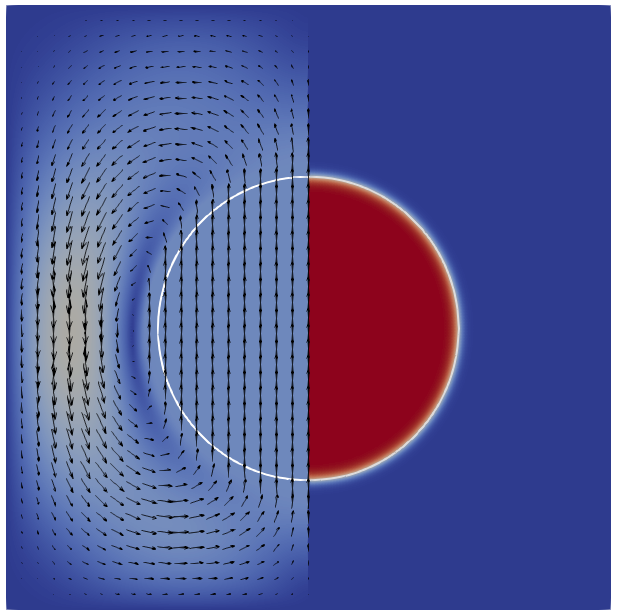}\\
\hspace*{2.3cm}(a)\hfill(b)\hfill(c)\hspace*{2.3cm}
\caption{Velocity (left side) and level-set (right side) at (a) moving down (b) rebound (c) moving up. The white contour depicts the interface where the value of level-set is equal to 0.5.}\label{fig_velocity}
\end{center}
\end{figure}

\begin{figure}[!h]
\begin{center}
\includegraphics[width=9cm]{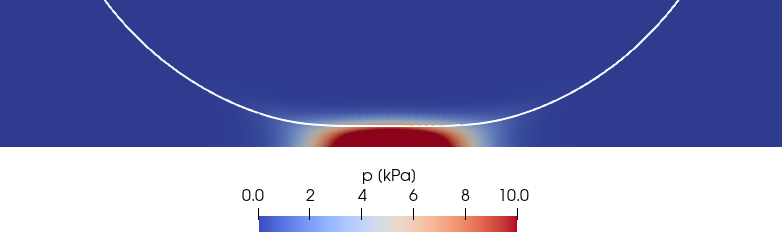}
\caption{Pressure in the fluid at the time of rebound.}\label{fig_pressure}
\end{center}
\end{figure}

\Cref{fig_velocity} shows the snapshots of the velocity and level-set fields at three time instances: moving down (panel a), during the rebound (panel b) and moving up (panel c). Since the viscosity considered is relatively high, the process is dissipative and the velocity magnitude is gradually decreasing with time. We remark that contact between the elastic ball and the wall never takes place -- it is indeed prevented by the development of a high-pressure region around the point where contact would normally be expected, see \Cref{fig_pressure}. The formation of this hydrodynamic pressure spike then facilitates the rebound. 

It is also worth noting how, in \Cref{fig_pressure}, the ball gets deformed, with the ``impacting'' face becoming very flat during the rebound phase. We fitted the shape of the interface at the bottom of the ball with the function
\[
y = d_1 + d_2 |x|^a,
\]
where, for simplicity, we fixed $d_2=1/(2R)=2.5$.  This choice of $d_2$ is optimal for a circle of radius $R$. The dependence of the exponent $a$ on $h$ is shown in \Cref{fig_exponent}, which demonstrates the significant flattening during the rebound, that is, we observe larger values of $a$ as the distance $h$ approaches its minimal value.

\begin{figure}[!h]
\begin{center}
\includegraphics[width=8cm]{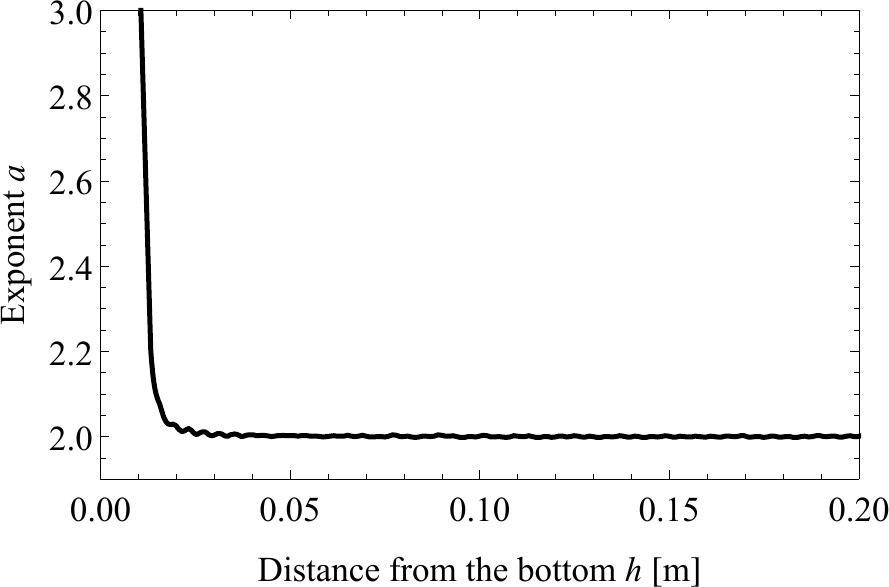}
\caption{The dependence of the exponent $a$ on $h$.}\label{fig_exponent}
\end{center}
\end{figure}

\subsection{A comparison of the two models} \label{sec:ODE-FEM_comparison}
Let us now compare in detail the numerical simulations performed for the reduced (effectively deformable) ODE model with the finite element solutions obtained for the full FSI problem. Throughout the section, we refer in particular to \Cref{fig:miracle}, where we display the distance to the wall $h$ as a function of time for both models and for several values of the viscosity parameter. It is worth noting that for the FSI model $h$ is defined as the distance between the wall and $0.5$-level set. Observe that with the choice of parameters made \Cref{sec:ODE_results}, the match between the two solutions is satisfactory in terms of the duration of the rebound phase and also regarding the mean body velocity after the rebound. 

A comment on the oscillatory behavior of $h$ after rebound in the ODE solutions is in order. While the oscillations are unmistakably due to the internal spring-mass system, which is effectively undamped for low viscosity values, it is worth mentioning that somewhat similar ``free oscillations'' were also observed in the simulations performed for the full FSI model for sufficiently small values of the shear modulus $G$ (not included in this paper). This can be seen as an indication that this particular feature of the ODE model should not be a-priori regarded as completely non-physical. 

Next, we note that the vertical offset of the solutions from the $x$-axis suggests that the finite element solutions bounce off at greater distances from the wall when compared to solutions of the ODE model. To some extent, this can be attributed to the effect of the level set approximation, i.e.\@ the diffuse interface between the ball and the fluid in the FSI model. Indeed, the fact that the material parameters are ``smeared'' over the diffuse interface of thickness $\e$ (see \eqref{diffuse_interface_tanh}) poses certain limitations on the minimal distance that (the $0.5$-level set of) the deformable structure may reach. Despite this issue, we are confident that for our choices of $\e$, spatial resolution, and high-enough viscosity $\mu$, the rebound due to the pressure singularity is not a mere artifact of the diffuse interface approach.

\section{Summary and concluding remarks}
\label{ref:conclusion}
In this paper, we investigate how serious and physically relevant is the so called no-contact paradox, that is, the absence of body-body and body-wall topological contact for elastic particles in an incompressible Stokes fluid with no-slip boundary conditions imposed on all boundaries. We were driven partially by the question whether the no-slip boundary condition in fluid-structure interaction problems must be avoided and branded as non-physical, or whether it can be redeemed somehow. We believe that we have provided an affirmative answer to the latter question, as we have shown that even in the absence of topological contact between an elastic body in motion towards a rigid wall, an effective rebound can be achieved, which is physical in the sense that it withstands the vanishing viscosity limit.

It is known and it has been proved rigorously that neither topological contact nor rebound are possible for perfectly rigid bodies (for a demonstration of this phenomenon see, for example, \Cref{Fig:rigid_FEM}; see also \Cref{sec:rigid-body} and \Cref{rigid-no-rebound}). On the other hand, the inclusion of elastic deformations of the solid bodies has been hypothesized as a promising ingredient towards obtaining a physical rebound. We tried to follow this path, yet, to simplify the notoriously difficult fluid-structure interaction problem, we devised a simplified ODE model (see \Cref{red-model-subsec}) which captures the features that we believe to be essential. 

The model comprises a ball (immersed in an incompressible Stokes fluid with no-slip boundary conditions prescribed both on the boundary of the container and on the fluid-solid interface) which is moving towards a rigid wall. As a simplified model of elasticity, we introduced a single scalar internal parameter $\xi$, which can be visualized as the elongation of a spring attached to a certain mass within the ball (see \Cref{fig-spring-mass-model}). In this setting, when the ball is subjected to the drag and internal push-and-pull from the spring, we have proved that contact cannot happen in finite time for any value of the viscosity parameter, and moreover that there is no rebound in the vanishing viscosity limit (see \Cref{sec:rigid-body}; see also \Cref{mainNB} and \Cref{cor:1}). In view of this fact, we conjectured that the internal mechanical energy storage alone is not a sufficient mechanism to ensure particle rebound.

As a next step in our analysis, we have investigated how allowing for deformations of the solid body changes the picture. This was achieved by coupling the internal deformation parameter with the drag formula. As a model case, we considered a one-parameter family of graphs describing the near-to-contact shape of the solid body by a general power function of the form
$$
y = h + c |x|^\alpha,
$$
where $h$ is the distance of the closest point to the wall, while $c$ and $\alpha$ are parameters possibly depending on the elastic deformation, i.e., we take $c = c(\xi)$ and $\alpha = \alpha(\xi)$. We have derived the corresponding parameterization of the drag force exerted by the fluid on the ball for such ``deformed'' configurations as they approach the wall (see \Cref{app-estimates-stokes}). These formulas are consistent with the standard lubrication (Reynolds') theory (see \Cref{drag-reynolds-app} and \Cref{appendix-Rey}) and read
\[
\mathcal{D}(h, \xi) \sim h^{-\frac{3\alpha(\xi)}{1+\alpha(\xi)}} \quad \text{ if } N = 2, \hspace{1cm} \mathcal{D}(h, \xi) \sim h^{\frac{1-3\alpha(\xi)}{1+\alpha(\xi)}} \quad \text{ if } N = 3.
\]
It is worth noting that in the context of the standard Hertz theory of contact, the shape of the solid body does not change dramatically in the sense that ``spheres deform to ellipses'', so that the shape exponent $\alpha$ remains unaltered. Inspired however by real-world observations, where much more dramatic changes in the ``flatness'' of an impacting body are often observed, we relaxed the assumption of Hertz theory that $\alpha$ is constant and allowed it to change according to the elastic deformation described by $\xi$. 

Surprisingly, this appears to be the key missing ingredient -- the feature that allows to reproduce a physically meaningful rebound. Indeed, in \Cref{sec:math} we have proved the possibility of a rebound that withstands the vanishing viscosity limit. Furthermore, our proofs are supplemented with numerical simulations of the ODE system (see \Cref{fig:giovanni}). It is worth noting that the reduced model can predict rebound while incorporating at the same time the defining feature of our problem, that is, the lack of topological contacts. This is a direct consequence of the fact that, in view of the imposed no-slip boundary conditions, the drag force exerted by the fluid blows up as the distance of the body from the wall approaches zero. 

Not only the ODE model admits a rigorous analysis of the effective rebound process, but, despite its apparent simplicity, it also shows a striking capability to reproduce qualitative characteristics of the rebound process when compared to finite element simulations of the full fluid-structure interaction problem (see \Cref{fig:miracle}). This gives us the confidence to consider the rigorously proved result for the ODE model as a reliable proof-of-concept for the general fluid-structure interaction problem outlined in \Cref{section:full-FSI} and to strengthen our main conjecture from the Introduction, i.e.\@, the claim that {\em a qualitative change in the flatness of the solid body as it approaches the wall, together with some elastic energy storage mechanism within the body, allows for a physically meaningful rebound even in the absence of topological contact.}

\subsection*{Acknowledgements}
The work of the authors was supported by the University Centre of Charles University. The research of G.\@ Gravina, S.\@ Schwarzacher and K.\@ T\r{u}ma was partially funded by the Czech Science Foundation (GA\v{C}R) under Grant No.\@ GJ19-11707Y. G.\@ Gravina and  S.\@ Schwarzacher further acknowledge the support of the Primus Programme of Charles University under Grant No.\@ PRIMUS/19/SCI/01. The authors would also like to thank B.\@ Bene\v{s}ova, M.\@ Kampschulte, and M.\@ Hillairet for helpful discussions on the subject.

\appendix
\section{Drag force estimates}
The material in this appendix is meant to complement our treatment of the drag force in \Cref{drag-section} by providing a proof of the analytical estimates and a derivation of the drag force predicted by the lubrication approximation theory. 

\subsection{The drag force based on the variational formulation}
\label{app-estimates-stokes}
We begin this first part of the appendix with the proof of \Cref{2&3D-LB}. The argument we present here is directly adapted from the proof of Lemma 3 in \cite{MR2354496}.

\begin{proof}[Proof of \Cref{2&3D-LB}]
We divide the proof into two steps.
\newline
\textbf{Step 1:} Assume first that $N = 2$. Then, by a density argument, it is enough to show that for every $\bm{v} \in C_c^{\infty}(\RR^2_+; \RR^2) \cap V_h$
\[
c_1 \le h^{\frac{3 \alpha}{1 + \alpha}} \|\nabla \bm{v} \|_{L^2}^2, 
\]
where $c_1$ is a positive constant independent of $\bm{v}$. To see this, using the notation introduced in \eqref{gh+}, we define
\[
\mathcal{F}_h(\delta) \coloneqq \{(x_1, x_2) : |x_1| < \delta, 0 < x_2 < g(|x_1|)\} \subset \mathcal{F}_h
\] 
and integrate $\di \bm{v} = 0$ in $\mathcal{F}_h(\delta)$ to obtain
\[
\int_{\partial \mathcal{F}_h(\delta) \cap \partial \B_h} \bm{v} \cdot \bm{n} \,d\mathcal{H}^1 = - \int_{\mathcal{F}_h(\delta) \cap \{|x_1| = \delta\}} \bm{v} \cdot \bm{n} \, d\mathcal{H}^1.
\]
Since $\bm{v} = \bm{e}_2$ on $\partial \B_h$, we see that
\[
\mathcal{L} \coloneqq \int_{\partial \mathcal{F}_h(\delta) \cap \partial \B_h} \bm{v} \cdot \bm{n} \,d\mathcal{H}^1 = \int_{\partial \mathcal{F}_h(\delta) \cap \partial \B_h} n_2 \,d\mathcal{H}^1 = 2\delta,
\]
where the last equality is obtained via a direct computation, parameterizing the domain of integration. Similarly, but also using the fact that $g(\delta) = g(- \delta)$, we obtain that
\[
\mathcal{R} \coloneqq \int_{\mathcal{F}_h(\delta) \cap \{|x_1| = \delta\}} \bm{v} \cdot \bm{n} \, d\mathcal{H}^1 = \int_0^{g(\delta)} \left(v_1(\delta, x_2) - v_1(- \delta, x_2)\right)\,dx_2,
\] 
and an application of H\"older's and Poincar\'e's inequalities yields
\begin{align*}
\int_0^{g(\delta)}|v_1(\delta, x_2) - v_1(- \delta, x_2)|\,dx_1 & \le g(\delta)^{1/2}\|v_1(\delta, \cdot) - v_1(- \delta, \cdot)\|_{L^2((0, g(\delta)); \RR^2)} \\
& \le g(\delta)^{3/2}\left\|\frac{\partial v_1}{\partial x_2}(\delta, \cdot) - \frac{\partial v_1}{\partial x_2}(- \delta, \cdot)\right\|_{L^2((0, g(\delta)); \RR^2)}.
\end{align*}
In turn, 
\[
2\delta = \mathcal{L} \le |\mathcal{R}| \le \sqrt{2} g(\delta)^{3/2} \left(\int_0^{g(\delta)}\left(|\nabla v_1(\delta, x_2)|^2 + |\nabla v_1(- \delta, x_2)|^2\right)\,dx_2\right)^{1/2}.
\]
Integrating the previous inequality over $\delta \in (0, r)$ yields
\begin{align*}
r^2 & \le  \sqrt{2}  \sup_{\delta \in (0, r)} \left\{ g(\delta)^{3/2}\right\} \int_0^r \left(\int_0^{g(\delta)}\left(|\nabla v_1(\delta, x_2)|^2 + |\nabla v_1(- \delta, x_2)|^2\right)\,dx_2\right)^{1/2}\,d\delta \\
& \le \sqrt{2}  \sup_{\delta \in (0, r)} \left\{ g(\delta)^{3/2}\right\} r^{1/2} \left(\int_0^r \int_0^{g(\delta)}\left(|\nabla v_1(\delta, x_2)|^2 + |\nabla v_1(- \delta, x_2)|^2\right)\,dx_2 d\delta \right)^{1/2} \\
& = \sqrt{2}  \sup_{\delta \in (0, r)}\left\{ g(\delta)^{3/2}\right\} r^{1/2} \|\nabla \bm{v}\|_{L^2},
\end{align*}
where in the second to last step we have used H\"older's inequality. Consequently, recalling that $g$ is given as in \eqref{gh+}, if we let $r = h^{1/(1 + \alpha)}$ we obtain
\begin{align*}
\frac{1}{\sqrt{2}} & \le \left(\frac{\sup\left\{h + \gamma \delta^{1 + \alpha} : \delta \in (0, r)\right\}}{r} \right)^{3/2} \|\nabla \bm{v} \|_{L^2} = (1 + \gamma)^{3/2}h^{\frac{3\alpha}{2(1 + \alpha)}}\|\nabla \bm{v} \|_{L^2}.
\end{align*}
The desired result readily follows.
\newline
\textbf{Step 2:} Now, assume that $N = 3$. Reasoning as in the previous step, but with the aid of cylindrical coordinates $(\delta, \theta, z)$, we readily deduce that
\[
\pi \delta^2 \le g(\delta)^{3/2} \delta \int_0^{2\pi} \left(\int_0^{g(\delta)}|\nabla \bm{v}|(\delta, \theta, z)|^2\,dz\right)^{1/2}\,d\theta
\]
holds for every $\bm{v} \in C_c^{\infty}(\RR^3_+; \RR^3) \cap V_h$. Thus, integrating the previous inequality over $\delta \in (0, r)$ and by means of H\"older's inequality, we get
\begin{align*}
\frac{\pi r^3}{3} & \le \sup_{\delta \in (0, r)} \left\{g(\delta)^{3/2} \delta^{1/2}\right\}  \int_0^r \delta^{1/2} \int_0^{2\pi} \left(\int_0^{g(\delta)}|\nabla \bm{v}|(\delta, \theta, z)|^2\,dz\right)^{1/2}\,d\theta d\delta \\
& \le \sqrt{2 \pi} \sup_{\delta \in (0, r)} \left\{g(\delta)^{3/2} \delta^{1/2}\right\} r^{1/2} \|\nabla \bm{v}\|_{L^2}.
\end{align*}
Therefore, setting once again $r = h^{1/(1 + \alpha)}$, we get
\[
\frac{1}{3}\sqrt{\frac{\pi}{2}} \le \sup_{\delta \in (0, r)} \left\{g(\delta)^{3/2}\right\} r^{-2} \|\nabla \bm{v}\|_{L^2} \le (1 + \gamma)^{3/2}h^{\frac{3 \alpha - 1}{2(1 + \alpha)}}\|\nabla \bm{v}\|_{L^2}.
\]
This concludes the proof.
\end{proof}

Next, we turn our attention to the proof of \Cref{2&3D-UB}, which we only sketch here. Recalling that by definition $D(h) = \min \left\{\J(\bm{u}; \mathcal{F}_h) : \bm{u} \in V_h\right\}$, the conclusions of \Cref{2&3D-UB} follow if we can exhibit a competitor, namely $\bm{w}_h \in V_h$, for which $\J(\bm{w}_h; \mathcal{F}_h)$ is bounded by the right-hand side of \eqref{drag-UB}. To achieve this, one has to construct a velocity field which allows for the fluid to escape the aperture in between the solid body and the boundary of the container in a nearly optimal way. For $N = 2$, such a construction was carried out by G\'erard-Varet and Hillairet (see Section 4.1 and Proposition 8 in \cite{MR2592281}). Their argument is adapted from the analogous construction for a two-dimensional disk, due to Hillairet (see Section 4 in \cite{MR2354496}). The construction for $N = 3$ is due to Hillairet and Takahashi (see Section 3.1 in \cite{MR2481302}) for a sphere, and can be suitably modified for the more general shapes that we consider in this paper.

\subsection{The drag force based on the Reynolds approximation}
\label{appendix-Rey}
In this second part of the appendix, we are interested in approximating the drag force exerted on a particle immersed in a Newtonian fluid, which is moving towards a rigid wall, when both the wall and the fluid-solid interface are subjected to no-slip boundary conditions. Considering an axi-symmetric situation (as in \Cref{fig_velocity}), a good approximation can be found by calculating and integrating the pressure under the solid body, which indeed represents the major contribution to the drag force \cite{Leal1992}. The pressure profile can be estimated using the so-called lubrication, or Reynolds', approximation. This yields the following ODE (see eq.\@ (7-256) in \cite{Leal1992} for $N = 2$; see eq.\@ (4.22) in \cite{Bassani1992} for $N = 3$):
\[
\frac{d}{dr}\left( r^{N - 2} g^3 \frac{dp}{dr} \right) = 12 \mu r^{N - 2} \dot{h} \hspace{1cm}N = 2, 3,
\]
where $r$ is the distance from the symmetry axis, and $g$ is defined as in \eqref{gh+}. Integrating the previous equation from 0 to $r'$, and recalling that by assumption the particle is axi-symmetric, yields 
\[
\frac{dp}{dr}(r') = 12 \mu \dot{h} \frac{r'}{g(r')^3(N - 1)}\hspace{1cm}N = 2, 3.
\]
Integrating now between $r$ and $R$ we obtain
\[
p(r) - p(R) = - \frac{12 \mu \dot{h}}{N - 1}\int_r^R \frac{r'}{g(r')^3}\,dr'.
\]
Assuming $R \gg 1$ then the pressure difference on the left-hand side corresponds to the actual dynamic pressure, which constitutes the main contribution to the drag force. Integrating the pressure difference over the surface of the solid body yields the (pressure contribution) to the drag force
\[
\tilde{\bm{F}}_{\operatorname{lub}} \coloneqq - \int_\Gamma \left(p(r) - p(R)\right) \bm{n} \,d\mathcal{H}^{N - 1},
\]
where $\bm{n}$ is the outer unit normal to $\Gamma$ (pointing inside the fluid). By symmetry, we only need to evaluate the vertical component of the force, since all other components are zero:
\[
\tilde{F}_{\operatorname{lub}} \coloneqq \tilde{\bm{F}}_{\operatorname{lub}} \cdot \bm{e}_N = - \int_\Gamma \left(p(r) - p(R)\right) \bm{n} \,d\mathcal{H}^{N - 1} \cdot \bm{e}_N.
\]
Letting $R \to \infty$ in the expression above yields 
\[
F_{\operatorname{lub}} \coloneqq \lim_{R \to \infty} \tilde{F}_{\operatorname{lub}} = -12\mu\dot{h} 
\left\{
\arraycolsep=3.4pt\def\arraystretch{2.2}
\begin{array}{ll}
\displaystyle 2\int_{0}^\infty\int_{r}^\infty \frac{r'}{g(r')^3} \,dr'dr & \text{ if } N = 2, \\
\displaystyle \pi \int_{0}^\infty \int_{r}^\infty \frac{r r'}{g(r')^3} \,dr'dr & \text{ if } N = 3,
\end{array}\right.
\]
which is the desired approximation.

\bibliographystyle{siam}
\bibliography{References}
\end{document}